\documentclass[12pt]{article}

\usepackage{amsmath, amssymb, amsthm, physics, url, mathrsfs, cite, xcolor, mathtools, dsfont, enumitem, tikz, adjustbox}

\mathtoolsset{showonlyrefs} 

\usetikzlibrary{fit, intersections, calc, angles, scopes}
\tikzset{every picture/.style={line width=0.75pt}} 

\definecolor{darkgreen}{RGB}{0, 90, 0}
\definecolor{darkred}{RGB}{234, 5, 35}
\definecolor{navyblue}{RGB}{0, 0, 128}

\usepackage[colorlinks, citecolor=darkgreen, linkcolor=darkred, urlcolor=navyblue, bookmarks=false,
	pdfauthor={Aobo Chen},
	pdftitle={Boundary Harnack principle on uniform domains},
	pdfkeywords={boundary Harnack principle, elliptic Harnack inequality, uniform domain, Dirichlet form}
	]{hyperref}

\usepackage[margin=1 in]{geometry}

\numberwithin{equation}{section}
\numberwithin{figure}{section}

\newtheorem{theorem}{Theorem}[section]
\newtheorem{lemma}[theorem]{Lemma}
\newtheorem{proposition}[theorem]{Proposition}

\theoremstyle{definition}
\newtheorem{definition}[theorem]{Definition}
\newtheorem{remark}[theorem]{Remark}
\newtheorem{notation}[theorem]{Notation}

\newcommand{\supp}{\mathrm{supp}}
\newcommand{\diam}{\mathrm{diam}}
\newcommand{\dist}{\mathrm{dist}}
\DeclareMathOperator*{\esssup}{\,ess\, sup}
\DeclareMathOperator*{\essinf}{\,ess\, inf}
\newcommand{\EHI}{\hyperlink{EHI}{\mathrm{EHI}}}
\newcommand*{\dif}{\mathop{}\!\mathrm{d}}
\newcommand{\one}{\mathds{1}}
\newcommand{\Capacity}{\operatorname{Cap}}

\newcounter{cnstcnt}
\newcommand{\newconstant}{%
\refstepcounter{cnstcnt}%
\ensuremath{C_{\thecnstcnt}}}
\newcommand{\oldconstant}[1]{\ensuremath{C_{\ref*{#1}}}}

\font\titlefont=cmbx12 scaled 1400
\title{\titlefont Boundary Harnack principle on uniform domains}

\author{Aobo Chen\thanks{Research partially supported by China Scholarship Council.}}

\date{\today}

\begin{document}

\maketitle

\vspace{-0.7cm}

\begin{abstract}
We present a proof of scale-invariant boundary Harnack principle for uniform domains when the underlying space satisfies a scale-invariant elliptic Harnack inequality. Our approach does not assume the underlying space to be geodesic. Additionally, the existence of Green functions is also not assumed beforehand and is ensured by a recent result from M. T. Barlow, Z.-Q. Chen and M. Murugan.


\end{abstract}

\section{Introduction}\label{s-intro}

The \emph{boundary Harnack principle} (BHP) is a property of a domain that controls the ratio of two positive harmonic functions near some part of the domain where both functions vanish. It is an important tool to help understand the boundary behaviour of harmonic functions.

The BHP for Lipschitz domains was first obtained independently by Ancona \cite{Anc78}, Dahlberg \cite{Dah77} and Wu \cite{Wu78}, and has been extended in several ways. For example, Aikawa \cite{Aik01} gave a BHP for length uniform domains in Euclidean space with standard Laplacian; Gyrya and Saloff-Coste \cite{GSC11} proved a BHP for inner uniform domains in a measure metric length space with a Dirichlet form that satisfies the standard parabolic Harnack inequality with walk dimension $2$. These results are further extended by Lierl \cite{Lie15} to inner uniform domains on length space with a Dirichlet form that satisfies parabolic Harnack inequality with anomalous space-time scaling, which is the case for diffusions on some fractals. Another recent progress of BHP is \cite{BM19}, where the authors there show that BHP for inner uniform domains is implied purely by elliptic information, namely the elliptic Harnack inequality and existence of Green function.

One limit in the definitions of being \emph{length uniform domain} and \emph{inner uniform domain} arises from their requirement for the existence of a non-trivial rectifiable curve, which is not preserved under some quasi-symmetric changes of metric such as snowflake transform. However, the property of being a \emph{uniform domain} that we considered in this work (see Definition \ref{d-ud}) is preserved under quasi-symmetric transforms of metric. Since quasi-symmetric transformation of metrics recently plays a significant role in helping us understand heat kernel and Harnack inequalities \cite{Kig12, BCM22, KM23b}, we adapt the notion of uniform domain and expect BHP would also hold under this weaker definition.

The main result of this paper is as follows. See Section \ref{s-mmd-g} and \ref{s-ud} for detailed definitions and notations.

\begin{theorem}[BHP]\label{t-bhp}
    Let $(\mathcal{X},d)$ be a complete metric doubling space, and let $m$ be a Radon measure on $\mathcal{X}$ with full support. Let $(\mathcal{E},\mathcal{F})$ be a strongly local symmetric regular Dirichlet form on $L^{2}(\mathcal{X},m)$. Suppose that $(\mathcal{X},d,m,\mathcal{E},\mathcal{F})$ satisfies the elliptic Harnack inequality. Let $U\subsetneq\mathcal{X}$ be an $A$-uniform domain. Then there exist $A_{0}, C,C_{0}\in(1,\infty)$ such that for all $\xi\in\partial U$, for all $0<r<C_{0}^{-1}\diam(U,d)$ and any two non-negative functions $u, v$ that are $\mathcal{E}$-harmonic in $B_{U}(\xi, A_{0}r)$ with Dirichlet boundary condition along $\partial U\cap B_{U}(\xi,2A_{0}r)$, we have \begin{equation}
        \esssup_{x\in B_{U}(\xi,r)}\frac{u(x)}{v(x)}\leq C \essinf_{x\in B_{U}(\xi,r)}\frac{u(x)}{v(x)}.
    \end{equation}
    where $C$ depends only on $A$, the metric $d$ and the constants that appear in elliptic Harnack inequality; $C_{0}$ and $A_{0}$ depend only on $A$ and the metric $d$.
\end{theorem}

{As an application, the above boundary Harnack principle can be used to identify the Martin boundary of a bounded uniform domain as its topological boundary \cite[Corollary 3]{Aik01}. In a recent work of Kajino and Murugan \cite{KM23a}, the boundary Harnack principle shown in Theorem \ref{t-bhp} together with Moser’s iteration techniques \cite[Lemma 3.10]{KM23a} ensure the continuous extension of \emph{Naïm kernel} and \emph{Martin kernel} up to the topological boundary of uniform domains, in the context of metric measure space that need not contain any non-constant rectifiable curves. Since the {Naïm kernel} on the topological boundary is shown to be the jump density of the trace process with respect to harmonic measure \cite[Theorem 5.12]{KM23a}, such extension plays an important role in obtaining the heat kernel estimate of trace process on the boundary.}

The contents of this paper are as follows. In Section \ref{s-mmd-g} we give the definitions of some terminologies of Dirichlet form theory and some facts about Green functions. Section \ref{s-ud} reviews the definition and basic properties of uniform domains. Section \ref{s-proof} gives the proof of Theorem \ref{t-bhp}. We will follow Aikawa's method in \cite{Aik01}, which reduces BHP to the estimates of Green functions.

\begin{notation}
	Throughout this paper, we use the following notation.
\begin{enumerate}[label=\textup{(\roman*)},align=right,leftmargin=*,topsep=5pt,parsep=0pt,itemsep=2pt]
		\item  The symbols $\subset$ and $\supset$ for set inclusion
		\emph{allow} the case of the equality.
 		\item Let $X$ be a non-empty set. We define $\one_{A}=\one_{A}^{X}\in\mathbb{R}^{X}$ for $A\subset X$ by
	 \[\one_{A}(x):=\one_{A}^{X}(x):= \begin{cases}
	 	1 & \mbox{if $x \in A$,}\\
	 	0 & \mbox{if $x \notin A$.}
	 \end{cases} \]
	 
		\item Let $X$ be a topological space. We set
		$C(X):=\{f\mid\textrm{$f:X\to\mathbb{R}$, $f$ is continuous}\}$ and
		$C_c(X):=\{f\in C(X)\mid\textrm{$X\setminus f^{-1}(0)$ has compact closure in $X$}\}$.
		\item In a metric space $(X,d)$, $B(x,r)$ is the open ball centered at $x \in X$ of radius $r>0$.
	  For a subset $A \subset X$, we use the notation $B_A(x,r) := A \cap B(x,r)$ for $x \in X, r>0$.
	  \item Given a ball $B:=B_U(x,r)$ (respectively $B:=B(x,r)$) and $K>0$, by $KB$ we denote the ball $B_U(x,Kr)$ (resp. $B(x,Kr)$).
	  \item For a set $A \subset X$, we write $\overline{A}, A^\circ, \partial A= \overline{A} \setminus A^\circ$ to denote its closure, interior and boundary respectively. Write $A_{\mathrm{diag}}=\{(x,x)\in X\times X: x\in A\}$.
	  \item For a non-empty open set $U \subsetneq X$, set \[\delta_{U}(x):=\dist(x,{X}\setminus U)=\inf\{d(x, y):y\in{X}\setminus U\} \quad\text{for all }x\in U.\]
	\end{enumerate}
\end{notation}

\section{Metric measure Dirichlet space and Green functions}\label{s-mmd-g}

Throughout this paper, we consider a complete metric space $(\mathcal{X},d)$, and a Radon measure $m$ on $(\mathcal{X},d)$ with full support, i.e., a Borel measure $m$ on $\mathcal{X}$ which is finite on any compact subset of $\mathcal{X}$ and strictly positive on any non-empty open subset of $\mathcal{X}$. We set $\diam(A, d)=\sup_{x, y\in A}d(x, y)$ for $A\subset\mathcal{X}$ ($\sup\emptyset:=0$). We will only consider the setting for a metric doubling space.

\begin{definition}
    A metric space $(\mathcal{X},d)$ is said to be \emph{metric doubling} if there exists $N\geq2$ such that every ball $B(x, R)$ can be covered by $N$ balls of radii $R/2$ for all $x\in \mathcal{X}$, $R>0$.
\end{definition}
\begin{remark}\label{r-dbct}
\begin{enumerate}[label=\textup{(\roman*)},align=left,leftmargin=*,topsep=5pt,parsep=0pt,itemsep=2pt]
 \item It is known that every complete metric doubling metric space is separable \cite{Ass83} and locally compact \cite[Exercise 10.17]{Hei01}.
\item An alternate but equivalent definition for metric doubling spaces is used in this work. A metric space $(\mathcal{X},d)$ is metric doubling if and only if, there exists some $\alpha>0$ with the property that every ball of radius $r$ has at most $C_{D}\epsilon^{-\alpha}$ disjoint points of mutual distance at $\epsilon r$, for some $C_{D}\geq1$ independent of the ball. See \cite[Exercise 10.17]{Hei01}. We will fix these two constants $C_{D}$ and $\alpha$ in this paper and refer them as \emph{doubling constants}.
\end{enumerate}
\end{remark}

Let $(\mathcal{E},\mathcal{F})$ be a \emph{regular, symmetric Dirichlet form} on $L^{2}(\mathcal{X},m)$; that is, 
\begin{enumerate}[label=\textup{(\roman*)},align=left,leftmargin=*,topsep=5pt,parsep=0pt,itemsep=2pt]
    \item $\mathcal{F}$ is a dense linear subspace of $L^{2}(\mathcal{X},m)$;
    \item $\mathcal{E}$ is a non-negative definite symmetric bilinear form on $\mathcal{F}\times\mathcal{F}$;
    \item (Closedness) $\mathcal{F}$ is a Hilbert space with inner product $\mathcal{E}_{1}:=\mathcal{E}+\langle\cdot,\cdot\rangle_{L^{2}(\mathcal{X},m)}$;
    \item (Markovian) for every $f\in \mathcal{F}$, $h:=(0\vee f)\wedge1\in\mathcal{F}$ and $\mathcal{E}(h, h)\leq\mathcal{E}(f, f)$;
    \item (Regularity) $\mathcal{F}\cap C_{c}(\mathcal{X})$ is dense both in $(\mathcal{F},\sqrt{\mathcal{E}_{1}})$ and in $(C_{c}(\mathcal{X}),\lVert\cdot\rVert_{\sup})$.
\end{enumerate}
We say a regular, symmetric Dirichlet form $(\mathcal{E},\mathcal{F})$ is \emph{strongly local} if, in addition,
\begin{enumerate}
  \item[(vi)] $\mathcal{E}(f, g)=0$ for any $f, g\in\mathcal{F}$ with $\supp_{m}(f)$, $\supp_{m}(g)$ compact and $\supp_{m}(f-a\one_{\mathcal{X}})\cap \supp_{m}(g)=\emptyset$ for some $a\in\mathbb{R}$. Here $\supp_{m}(f)$ denotes the support of the measure $|f|\dif m$.
\end{enumerate}

We recall some analytic notations associated with a regular Dirichlet form.
\begin{definition} Given a regular Dirichlet form $(\mathcal{E},\mathcal{F})$ on $L^2(\mathcal{X}, m)$.
\begin{enumerate}[label=\textup{(\roman*)},align=left,leftmargin=*,topsep=5pt,parsep=0pt,itemsep=2pt]
\item An increasing sequence $\{F_{k}; k\geq1\}$ of closed subsets of $\mathcal{X}$ is said to be an \emph{$\mathcal{E}$-nest} if $\bigcup_{k>1}\mathcal{F}_{F_k}$ is $\sqrt{\mathcal{E}_1}$-dense in $\mathcal{F} , $ where $\mathcal{F} _{F_k}: = \{ f\in \mathcal{F} : f= $ 0 $m$-a.e. on $\mathcal{X}\backslash F_k\}$. 
\item A set $N\subset\mathcal{X}$ is said to be \emph{$\mathcal{E}$-polar} if there is an $\mathcal{E}$-nest $\{F_{k}; k\geq1\}$ so that $N\subset\mathcal{X}\backslash\bigcup_{k\geq1}F_k$.
\item A statement depending on $x\in A$ is said to hold \emph{$\mathcal{E}$-quasi-everywhere} ($\mathcal{E}$-q.e. in abbreviation) if there is an $\mathcal{E}$-polar set $N\subset A$ so that the statement is true for every $x\in A\setminus N$.  
\item A real-valued function $f$ is said to be in the \emph{extended Dirichlet space} $\mathcal{F}_e$ if there is an $\mathcal{E}$-Cauchy sequence $\{f_{k}; k\geq1\}\subset\mathcal{F}$ so that $\lim_{k\to\infty}f_k=f$ $m$-a.e. on $\mathcal{X}$, and we define $\mathcal{E}(f, f)=\lim_{k\to\infty}\mathcal{E}(f_{k}, f_{k})$.
\item A function $f$ is said to be $\mathcal{E}$-quasi-continuous on $\mathcal{X}$ if there is an $\mathcal{E}$-nest $\{F_{k}; k\geq1\}$ so that $f|_{F_{k}}\in C(F_{k})$ for every $k\geq1$, where $C(F_{k}):= \{u : F_{k}\rightarrow\mathbb{R}: u\text{ is continuous}\}$.
\item $(\mathcal{E},\mathcal{F})$ is said to be \emph{transient} if there exists a bounded $g\in L^1(\mathcal{X};m)$, called \emph{reference function}, that is strictly positive on $\mathcal{X}$ so that $$\int_{\mathcal{X}}|u(x)|g(x)\dif m(x)\leq\mathcal{E}(u,u)^{1/2}\quad\text{for every }u\in\mathcal{F}.$$
\end{enumerate}
\end{definition}
\begin{remark}\label{r-quasicont}
For a regular Dirichlet form $(\mathcal{E},\mathcal{F})$ on $L^2(\mathcal{X};m)$, every $f\in \mathcal{F}_e$ has an $m$-version that is $\mathcal{E}$-quasi-continuous on $\mathcal{X}$, which is unique up to an $\mathcal{E}$-polar set \cite[Theorem 2.3.4]{CF12}.  We always take a function $f\in \mathcal{F}_e$ to be represented by its $\mathcal{E}$-quasi-continuous version. Under this convention, the statements like ``$f=0\ \mathcal{E}\text{-q.e. on }\mathcal{X}$" will make sense for $f\in \mathcal{F}_e$.
\end{remark}
We define \emph{relative capacity} as follows. For any open subset $D\subset\mathcal{X}$ with non $\mathcal{E}$-polar complement, and any set $A\subset D$, define $$\Capacity_{D}(A):=\inf\{\mathcal{E}(f, f):f\in(\mathcal{F}^{D})_{e},\:f\geq1\ \mathcal{E}\text{-q.e. on }A,\}$$ where \[ (\mathcal{F}^{D})_{e}:=\left\{u\in \mathcal{F}_{e}: u=0\ \mathcal{E}\text{-q.e. on }\mathcal{X}\setminus D\right\}.\]
\begin{definition}[Local Dirichlet space]
    Let $V\subset D$ be open subsets of $\mathcal{X}$. Define the following function spaces:
    \begin{align*}
        \mathcal{F}^{0}(D)&:=\left\{f\in\mathcal{F}: f=0\ \mathcal{E}\text{-q.e. on }\mathcal{X}\setminus D \right\},\\
\mathcal{F}_{\mathrm{loc}}(D)&:=
  \Biggl\{f\in L^{2}_{\mathrm{loc}}(\mathcal{X},m)  \Biggm|
  \begin{minipage}{240pt}
For any relatively compact open subset $A$ of $D$,  there exists $f^{\#}\in \mathcal{F}$ such that $f^{\#}=f\ m$-a.e. on $A$.
\end{minipage}
\Biggr\},
\\ 
\mathcal{F}^{0}_{\mathrm{loc}}(D,V)&:=\Biggl\{f\in L^{2}_{\mathrm{loc}}(\mathcal{X},m)  \Biggm|
\begin{minipage}{240pt}
For any open subset $A$ of $\mathcal{X}$ that is relatively compact in $\overline{D}$ with $d(A,D\setminus V)>0$, there exists $f^{\#}\in \mathcal{F}^{0}(D)$ such that $f^{\#}=f\ m$-a.e. on $A$.
\end{minipage}
\Biggr\}.
    \end{align*}
\end{definition}

Denote $\{P_{t}\}$ as the strongly continuous symmetric contractive semigroup on $L^{2}(\mathcal{X},m)$ corresponding to $(\mathcal{E},\mathcal{F})$. By \cite[Theorem 7.2.1]{FOT11}, there is an $m$-symmetric continuous Hunt process $X=\left\{X_{t},t\geq0;\ \mathbb{P}^{x},x\in\mathcal{X}\right\}$ on $\mathcal{X}$ associated with $(\mathcal{E},\mathcal{F})$ in the sense that
\[P_{t}f(x)=\mathbb{E}^{x}\left[f(X_{t})\right],\ m\text{-a.e. }x\in\mathcal{X}\]
for all $f\in L^{\infty}(\mathcal{X},m)$ and every $t>0$. We write $\{\mathscr{F}_{t}\}$ for the minimum augmented admissible filtration of $X$(see \cite[p.397]{CF12} for definition). For a Borel subset $B \subset \mathcal{X}$, define \[\tau_{B}:=\inf \left\{t>0: X_t \notin B\right\}.\] Then $\tau_{B}$ is a $\{\mathscr{F}_{t}\}$-stopping time \cite[Theorem A.1.19]{CF12}. A set $\mathcal{N}\subset \mathcal{X}$ is said to be \emph{Borel properly exceptional} for $X$, if $\mathcal{N}$ is Borel measurable, $m(\mathcal{N})=0$ and \[\mathbb{P}^{x}(X_{t}\in (\mathcal{X}\cup\{\partial\})\setminus \mathcal{N} \text{ for all $t>0$})=1\text{ for every $x\in\mathcal{X}\setminus \mathcal{N}$}.\] 

\begin{remark}
If $D$ is an open subset of $\mathcal{X}$, then $(\mathcal{E},\mathcal{F}^{0}(D))$ is also a regular, strongly local symmetric Dirichlet form on $L^{2}(D, m|_{D})$ \cite[Theorem 3.3.9]{CF12}. The associated Hunt process $X^{D}$ of $(\mathcal{E},\mathcal{F}^{0}(D))$ is $X$ being killed upon leaving $D$, i.e.,\[X^{D}_{t}=\begin{cases}
        X_{t}\quad &t<\tau_{D};\\ \partial\quad &t\geq\tau_{D}.
    \end{cases}\]
Here $\partial$ is referred to as ``cemetery point" in the general theory of Markov process \cite[Appendix A]{CF12}.
\end{remark}
\begin{definition}
    Let $\Omega$ be an open subset of $\mathcal{X}$. \begin{enumerate}[label=\textup{(\roman*)},align=left,leftmargin=*,topsep=5pt,parsep=0pt,itemsep=2pt]
        \item We say a function $u$ is \emph{regular harmonic} in $\Omega$ with respect to the process $X$ if \[\mathbb{E}^{x}\left[|u(X_{\tau_{\Omega}})|\right]<\infty \text{ and }u(x)=\mathbb{E}^{x}\left[u(X_{\tau_{\Omega}})\right]\text{ for $\mathcal{E}$-q.e. }x\in \Omega.\]
        \item We say a function $u$ is \emph{$\mathcal{E}$-harmonic} in $\Omega$ if $u\in\mathcal{F}_{\mathrm{loc}}(\Omega)$ and \[\mathcal{E}(u, v)=0\text{ for every }v\in \mathcal{F}\cap C_{c}(\Omega).\]
        \item Let $V\subset \Omega$ be open subsets of $\mathcal{X}$. We say an $\mathcal{E}$-harmonic function $u$ in $V$ satisfies \emph{Dirichlet boundary condition along the boundary} $\partial \Omega\cap\overline{V}$ if $u\in\mathcal{F}^{0}_{\mathrm{loc}}(\Omega,V).$
    \end{enumerate}
\end{definition}

\begin{definition}[Elliptic Harnack inequality] We say that $(\mathcal{E}, \mathcal{F})$ satisfies the \emph{elliptic Harnack inequality} \hypertarget{EHI} with constants $C_{\mathrm{H}}<\infty$ and $\delta_{\mathrm{H}}\in(0,1)$, denoted $\hyperlink{EHI}{\mathrm{EHI}(C_{\mathrm{H}},\delta_{\mathrm{H}})}$, if for any ball $B(x, R) \subset \mathcal{X}$, and any non-negative function $u \in \mathcal{F}_{\mathrm{loc}}(B(x, R))$ that is $\mathcal{E}$-harmonic on $B(x, R)$, we have
$$
\esssup_{z \in B(x, \delta_{\mathrm{H}} R)} u(z) \leq C_{\mathrm{H}} \essinf_{z \in B(x, \delta_{\mathrm{H}} R)} u(z).
$$
We say that $\EHI$ holds if $\hyperlink{EHI}{\mathrm{EHI}(C_{\mathrm{H}},\delta_{\mathrm{H}})}$ holds for some $C_{\mathrm{H}}<\infty$ and $\delta_{\mathrm{H}} \in(0,1)$.
\end{definition}
\begin{remark}\label{r-escape}
\begin{enumerate}[label=\textup{(\roman*)},align=left,leftmargin=*,topsep=5pt,parsep=0pt,itemsep=2pt]
	\item If $\mathcal{X}$ is connected and $\EHI$ holds, then by \cite[Theorem 4.8, Proposition 3.2]{BCM22}, for any relatively compact open subset $\Omega\subset\mathcal{X}$ such that $\mathcal{X}\setminus\Omega$ is not $\mathcal{E}$-polar, we have $\mathbb{P}^{x}(\tau_{\Omega}<\infty)=1 $ for $\mathcal{E}$-q.e. $x\in \Omega$.
	\item By \cite[Theorem 5.4]{BCM22}, if $(\mathcal{X},d)$ is a complete metric space and $\EHI$ holds, then it is metric doubling if and only if it is \emph{relatively $K$ ball connected} \hypertarget{RBC(K)} ($\hyperlink{RBC(K)}{\mathrm{RBC}(K)}$ in abbreviation) for some $K\geq2$, i.e., for each $\varepsilon\in(0,1)$, there exists an integer $N=N(\varepsilon)\geq1$ such that if $x_0 \in\mathcal{X}$, $R>0$ and $x,y \in B(x_0,R)$, then there exists a chain of balls $B(z_i,\varepsilon R)$ for $i = 0,...,N$ such that $z_0 = x, z_N = y, B(z_i,\varepsilon R) \subset B(x_0,KR)$ for each $i = 0,...,N$ and $d(z_{i-1},z_i) < \varepsilon R$ for $1\leq i\leq N$. This relatively $K$ ball connected property will be used when we estimate Green functions (see Theorem \ref{t-g-cp} below).
\end{enumerate}
\end{remark}

\begin{definition}[Harnack chain] Let $U \subsetneq \mathcal{X}$ be a connected open set and $M\geq 1$. For $x, y \in$ $U$, an $M$-Harnack chain from $x$ to $y$ in $U$ is a sequence of balls $B_1, B_2, \ldots, B_n$ each contained in $U$ such that $x \in M^{-1} B_1, y \in M^{-1} B_n$, and $M^{-1} B_i \cap M^{-1} B_{i+1} \neq \emptyset$, for $i=1,2, \ldots, n-1$. The number $n$ of balls in a Harnack chain is called the length of the Harnack chain.
\end{definition}
\begin{remark}
    Suppose that $(\mathcal{E}, \mathcal{F})$ satisfies $\hyperlink{EHI}{\mathrm{EHI}(C_{\mathrm{H}},\delta_{\mathrm{H}})}$ . If $u$ is a non-negative continuous $\mathcal{E}$-harmonic function on a domain $U$, and if there is a $\delta_{\mathrm{H}}^{-1}$-Harnack chain from $x_{1}$ to $x_{2}$ whose length is less than $L(x_{1},x_{2};\delta_{\mathrm{H}}^{-1})$, then
\begin{equation}\label{e-hccp}
    C_{\mathrm{H}}^{-L\left(x_1, x_2; \delta_{\mathrm{H}}^{-1}\right)} u\left(x_1\right) \leq u\left(x_2\right) \leq C_{\mathrm{H}}^{L\left(x_1, x_2 ; \delta_{\mathrm{H}}^{-1}\right)} u\left(x_1\right).
\end{equation}
\end{remark}
The next proposition gives an estimate on the length of Harnack chain using metric doubling condition.
\begin{proposition}[Length of Harnack chain]\label{p-hc-p}Let $U$ be a domain in $\mathcal{X}$. Let $x,y\in U$ and $\gamma$ be a continuous curve in $U$ from $x$ to $y$. Assume that  $\delta_{U}(z)\geq\delta>0$ for all $z\in\gamma$. Then for any $M>1$ and any $r\in(0,\delta)$, there exists an $M$-Harnack chain $\{B_{j}=B(z_{j},r)\}$ from $x$ to $y$ in $U$, with $z_{j}\in\gamma$ and its length less than $C_{D}(1+r^{-1}M\diam(\gamma))^{\alpha}$, where $C_{D}$ and $\alpha$ are the doubling constants.
\end{proposition}
\begin{proof}
   Let $M>1$ and $r\in(0,\delta)$. Take a maximum $M^{-1}r$ separated set in $\gamma$, say $\{z_{j}\}_{j\in J}$, which exists by Zorn's Lemma. By definition, $\bigcup_{j\in J}B(z_{j},M^{-1}r)$ covers $\gamma$ and the balls $B(z_{j},(2M)^{-1}r)$, $j\in J$ are mutually disjoint. By the metric doubling condition of $(\mathcal{X},d)$, there exist some $C_{D}$ and $\alpha>0$ such that $|J|\leq C_{D}(1+r^{-1}M\diam(\gamma))^{\alpha}$. We can relabel $z_{j}$ so that $x\in B(z_{1},M^{-1}r)$, $y\in B(z_{|J|},M^{-1}r)$ and $B(z_{j},M^{-1}r)\cap B(z_{j+1},M^{-1}r)\neq\emptyset$ for all $j=1,2,\ldots,|J|-1$. Thus $\{B_{j}=B(z_{j},r)\}_{j=1}^{|J|}$ forms an $M$-Harnack chain from $x$ to $y$.
\end{proof}
The next two theorems are taken from \cite[Section 4, Section 5]{BCM22}, which ensure the existence of \emph{regular Green functions} and some related estimates under $\EHI$. \begin{theorem}[Existence of Green function, {{\cite[Theorem 4.6, Theorem 4.8]{BCM22}}}]\label{t-g-ex} Let $(\mathcal{X},d)$ be a complete metric doubling space and $(\mathcal{X},d, m,\mathcal{E},\mathcal{F})$ satisfies $\EHI$. Then $(\mathcal{E},\mathcal{F})$ has regular Green function, in the sense that, for any bounded, non-empty open set $D\subset\mathcal{X}$ whose complement $\mathcal{X}\setminus D$ is non $\mathcal{E}$-polar, there exists a non-negative $\mathcal{B}(D\times D)$-measurable function $g_{D}(x, y)$ on $(D\times D)\setminus D_{\mathrm{diag}}$ with the following properties:
    \begin{enumerate}[label=\textup{(\roman*)},align=left,leftmargin=*,topsep=5pt,parsep=0pt,itemsep=2pt]
        \item\label{lb.g-ex1} {\normalfont{(Symmetry)}} $g_D(x, y)=g_D(y, x)$ for all $(x, y) \in (D \times D) \setminus D_{\mathrm{diag}}$.
\item\label{lb.g-ex2}  {\normalfont{(Continuity)}} $g_D(x, y)$ is $[0,\infty)$-valued and jointly continuous in $(x, y) \in (D \times D) \setminus D_{\mathrm{diag}}$.
\item\label{lb.g-ex3} {\normalfont{(Occupation density)}} There is a Borel properly exceptional set $\mathcal{N}_D$ of $\mathcal{X}$ such that
$$
\mathbb{E}^x \left[\int_0^{\tau_D} f\left(X_s\right) \dif s\right]=\int_D g_D(x, y) f(y) \dif m(y) \quad \text { for every } x \in D \setminus \mathcal{N}_D .
$$
for any $f \in \mathcal{B}_{+}(D)$.
\item\label{lb.g-ex4} {\normalfont{(Harmonicity)}} For any fixed $y \in D$, the function $x \mapsto g_{D}(x, y)$ is in $\mathcal{F}_{\text {loc }}^{D \setminus\{y\}}$ and for any open subset $\Omega$ of $D$ with $y\notin\overline{\Omega}$, $x \mapsto g_D(x, y)$ is regular harmonic in $\Omega$ with respect to $X^{\Omega}$.
\item\label{lb.g-ex5}  {\normalfont{(Maximum principles)}} If $\Omega$ is a relative compact open subset of $D$ and  $x_0 \in \Omega$, then
$$
\inf _{\Omega \backslash\left\{x_0\right\}} g_D\left(x_0, \cdot\right)=\inf _{\partial \Omega} g_D\left(x_0, \cdot\right), \quad \sup _{D \backslash \Omega} g_D\left(x_0, \cdot\right)=\sup _{\partial \Omega} g_D\left(x_0, \cdot\right).
$$
\end{enumerate}
\end{theorem}
\begin{remark}
In view of Theorem \ref{t-g-ex}-\ref{lb.g-ex2}, \ref{lb.g-ex3} and \cite[Remark 2.7-(ii), Proposition 2.9-(iii)]{BCM22}, we infer for any fixed $y \in D$, the function $x \mapsto g_{D}(x, y)$ is $\mathcal{E}$-harmonic in $D \setminus\{y\}$.
\end{remark}
\begin{theorem}[Comparison of Green function,{{\cite[Section 5]{BCM22}}}]\label{t-g-cp} Let $(\mathcal{X},d)$ be a complete metric doubling space satisfies $\hyperlink{RBC(K)}{\mathrm{RBC}(K)}$ for $K\geq2$, and $(\mathcal{X},d, m, \mathcal{E},\mathcal{F})$ satisfies $\hyperlink{EHI}{\mathrm{EHI}(C_{\mathrm{H}},\delta_{\mathrm{H}})}$.
\begin{enumerate}[label=\textup{(\roman*)},align=left,leftmargin=*,topsep=5pt,parsep=0pt,itemsep=2pt]
\item\label{lb.cp1}  There exists $\newconstant\label{c1}=\oldconstant{c1}\left(K, C_{\mathrm{H}},\delta_{\mathrm{H}}\right)>1$ such that for all open sets $D$ in $\mathcal{X}$ whose complement $\mathcal{X}\setminus D$ is not $\mathcal{E}$-polar, and for all $x \in \mathcal{X}, r>0$ that satisfy $B\left(x, (1+2K) r\right) \subset D$,
\begin{equation}\label{e-hg}
\max_{y\in\partial B(x,r)}g_{D}(x,y)\leq \oldconstant{c1}\min_{y\in\partial B(x,r)}g_{D}(x,y),
\end{equation}
\begin{equation}\label{e-g-cap}
\min_{y\in\partial B(x,r)}g_{D}(x,y)\leq \Capacity_{D}(B(x,r))^{-1}\leq \oldconstant{c1}\min_{y\in\partial B(x,r)}g_{D}(x,y)
\end{equation}

\item\label{lb.cp2} For all $A_{1}>1$, there exists $\newconstant\label{c2}=\oldconstant{c2}\left(K, A_{1}, C_{\mathrm{H}},\delta_{\mathrm{H}}\right)>1$ such that for all open sets $D$ in $\mathcal{X}$ whose complement $\mathcal{X}\setminus D$ is not $\mathcal{E}$-polar, and for all $x \in \mathcal{X}, r>0$ that satisfy $B\left(x, (3+2K) r\right) \subset D$, we have for all $x_1, y_1, x_2, y_2 \in B\left(x, r\right)$ satisfying $d\left(x_i, y_i\right) \geq r /A_{1}$, $i=1,2$, that
$$g_D\left(x_1, y_1\right) \leq \oldconstant{c2} g_D\left(x_2, y_2\right).$$

\item\label{lb.cp3} For all $1<A_{1}\leq A_{2}<\infty$ and $0<s\leq r \leq \diam(\mathcal{X}) / (6(A_{2}\vee(9K)))$, there exist $\newconstant\label{c3}=\oldconstant{c3}\left(A_{1},A_{2},K, r/s, C_{\mathrm{H}},\delta_{\mathrm{H}}\right)>1$ and $\newconstant\label{c3a}=\oldconstant{c3a}\left(A_{1},A_{2},K, C_{\mathrm{H}},\delta_{\mathrm{H}}\right)>1$ such that for all $x \in \mathcal{X}$, \[\oldconstant{c3}^{-1}\Capacity_{B\left(x, A_{1}s\right)}\left(B\left(x, s\right)\right)\leq\Capacity_{B\left(x, A_{1}r\right)}\left(B\left(x, r\right)\right)\leq\oldconstant{c3}\Capacity_{B\left(x, A_{1}s\right)}\left(B\left(x, s\right)\right),\]
\[\Capacity_{B\left(x, A_{2}r\right)}\left(B\left(x, r\right)\right)\leq\Capacity_{B\left(x, A_{1}r\right)}\left(B\left(x, r\right)\right)\leq\oldconstant{c3a} \Capacity_{B\left(x, A_{2}r\right)}\left(B\left(x, r\right)\right).\]

\item\label{lb.cp4} For all $A_1 \geq 2$ and $r>0$, there exists $\newconstant\label{c4}=\oldconstant{c4}\left(A_1, C_{\mathrm{H}},\delta_{\mathrm{H}}\right)>1$ such that if $B(x_{0},2A_{1}r)^{c}$ is non-empty, then
$$
g_{B\left(x_{0}, r\right)}(x, y) \leq g_{B\left(x_{0}, A_{1} r\right)}(x, y) \leq \oldconstant{c4} g_{B\left(x_{0},r\right)}(x, y) \quad\text{for $x,y\in B(x_{0},r/(8K))$, $x\neq y$.}
$$
\end{enumerate}
\end{theorem}
\begin{proof} 
\begin{enumerate}[label=\textup{(\roman*)},align=left,leftmargin=*,topsep=5pt,parsep=0pt,itemsep=2pt]
\item[\ref{lb.cp1}] By \cite[Proposition 5.7]{BCM22}, we have \[\sup_{y\in D\setminus B(x,r)}g_{D}(x,y)\leq \oldconstant{c1} \inf_{y\in \overline{B(x,r)}\setminus\{x\} }g_{D}(x,y).\]
Combining with the maximum principles in Theorem \ref{t-g-ex}, $\sup_{y\in D\setminus B(x,r)}g_{D}(x,y)=\sup_{y\in \partial B(x,r)}g_{D}(x,y)$ and $\inf_{y\in \overline{B(x,r)}\setminus\{x\} }g_{D}(x,y)=\inf_{y\in \partial{B(x,r)} }g_{D}(x,y)$. As $\partial B(x,r)$ is compact and $g_{D}(x,\cdot)$ is continuous on $\partial B(x,r)$, we can replace $\sup$ and $\inf$ by $\max$ and $\min$, respectively. This gives \eqref{e-hg}. Inequalities \eqref{e-g-cap} is proved in \cite[Lemma 5.10]{BCM22}.
\item[\ref{lb.cp2}] The proof is a minor modification of the proof of \cite[Lemma 5.9]{BCM22}. A counting argument as in \cite[Lemma 5.9]{BCM22} shows that there exists $z_{0}\in B(x,r)$ such that $d(z_{0},w)\geq r/9$ for $w\in\{x_{1},x_{2},y_{1},y_{2}\}$. Then we can establish the estimate by applying \cite[Corollary 5.8]{BCM22} with $\delta=\min(1/(2A_{1}),1/18)$.
\item[\ref{lb.cp3}] See  \cite[Lemma 5.22, Lemma 5.23]{BCM22}.
\item[\ref{lb.cp4}] The first inequality holds by domain monotonicity of Green function. The latter inequality is an iteration of \cite[Lemma 5.18]{BCM22}.
\end{enumerate}
\end{proof}
We also need the following version of maximum principle for Green function. We notice that \cite[Lemma 4.12]{BM19} provides a proof by using the maximum principles in \cite[Lemma 4.1(ii)]{GH14}, which is stated under the assumption of the first Dirichlet eigenvalue being strictly positive. Instead, we use the regular harmonicity of Green functions provided in Theorem \ref{t-g-ex}.
 \begin{lemma}\label{l-g-mp2}
    Let $(\mathcal{X},d)$ be a complete metric doubling space and $(\mathcal{X},d,m,\mathcal{E},\mathcal{F})$ satisfies $\EHI$. Given a bounded, non-empty open set $D\subset\mathcal{X}$ whose complement $\mathcal{X}\setminus D$ is non $\mathcal{E}$-polar. Let $c_{0}\in(0,\infty)$, $y, y^{\star}\in D$ such that $B(y^{\star},(1+2K)r)$ is a relative compact open subset of $D$,  $y\notin \partial B(y^{\star},r)$ and \[g_{D}(y, x)\geq c_{0}g_{D}(y^{\star},x)\quad \text{for all }x\in\partial B(y^{\star},r).\]Then \[g_{D}(y, x)\geq c_{0}g_{D}(y^{\star},x)\quad \text{for all }x\in D\setminus(\{y\}\cup B(y^{\star},r)).\]
\end{lemma}
\begin{proof}
Since $\EHI$ holds and $\mathcal{X}\setminus D$ is not $\mathcal{E}$-polar, we have by Remark \ref{r-escape} that, \[ \mathbb{P}^{x}(\tau_{D}<\infty)=1\quad\text{ for $\mathcal{E}$-q.e. $x\in D$}.\]  

If $y\in B(y^{\star},r)$, then the function $x\mapsto g_{D}(y, x)- c_{0}g_{D}(y^{\star},x)$ is regular harmonic in $V:=D\setminus\overline{B(y^{\star},r)}$ with respect to $X^{D}$ by Theorem \ref{t-g-ex}-\ref{lb.g-ex4}. Thus for $\mathcal{E}$-q.e. $x\in V$,
\begin{align*}g_{D}(y, x)- c_{0}g_{D}(y^{\star},x)&=\mathbb{E}^{x}\left[g_{D}(y, (X^{D})_{\tau_{V}})- c_{0}g_{D}(y^{\star},(X^{D})_{\tau_{V}})\right]\\ &=\mathbb{E}^{x}\left[\left(g_{D}(y, X_{\tau_{V}})- c_{0}g_{D}(y^{\star},X_{\tau_{V}})\right)\one_{\{\tau_{V}<\tau_{D}\}}\right] \geq 0\end{align*} since for $\mathcal{E}$-q.e. $x\in V$, $X_{\tau_{V}}\in\partial B(y^{\star},r)$ $\mathbb{P}^{x}$-a.s. on the set $\{\tau_{V}<\tau_{D}\}$.

Suppose that $y\notin \overline{B(y^{\star},r)}$. We may assume that $\epsilon:=\min_{z\in\partial B(y^{\star},r)}g_{D}(y^{\star},z)>0$. Otherwise, by Theorem \ref{t-g-cp}-\ref{lb.cp1} and Theorem \ref{t-g-ex}-\ref{lb.g-ex5}, the function $g_{D}(y^{\star},\cdot)$ vanishes on $D\setminus B(y^{\star},r)$ and the conclusion holds automatically. We choose $s>0$ small enough so that $B(y,4s)\subset V$. Define \[h_{n}(x)=\frac{1}{m(B(y, s/n))}\int_{B(y, s/n)}g_{D}(z, x)\dif m(z),\quad n\geq1\]
By the uniform continuity of $g_{D}(\cdot,\cdot)$ on the compact subset $\overline{B(y,2s)}\times\partial B(y^{\star},r)\subset (D\times D)\setminus D_{\mathrm{diag}}$, $h_{n}\rightarrow g_{D}(y,\cdot)$ uniformly on $\partial B(y^{\star},r)$ as $n\rightarrow\infty$. If we write $f_{n}=(m(B(y, s/n)))^{-1}\one_{B(y, s/n)}$, by Theorem \ref{t-g-ex}-\ref{lb.g-ex3}, \[h_{n}(x)=\mathbb{E}^{x}\left[\int_{0}^{\tau_{D}}f_{n}(X_{t})\dif t\right]=\mathbb{E}^{x}\left[\int_{0}^{\infty} f_{n}(X^{D}_{t})\dif t \right]\quad\text{for $x\in D\setminus \mathcal{N}_{D}$}.\]By the strong Markov property \cite[Theorem A.1.22 (ii)]{CF12}, we have the following super-martingale property \[h_{n}(x)\geq \mathbb{E}^{x}\left[h_{n}((X^{D})_{\tau_{V}})\right]\quad\text{for $x\in D\setminus \mathcal{N}_{D}$}.\] For any $0<c_{1}<c_{0}$, since $\epsilon=\min_{z\in\partial B(y^{\star},r)}g_{D}(y^{\star},z)>0$, we have for any $x\in \partial B(y^{\star},r)$,
\begin{align*}
&h_{n}(x)-c_{1}g_{D}(y^{\star},x) \\ =\ & \left(h_{n}(x)- g_{D}(y,x)\right)+\left(g_{D}(y,x)-c_{0}g_{D}(y^{\star},x)\right)+\left((c_{0}-c_{1})g_{D}(y^{\star},x)\right)\\ \geq\ & h_{n}(x)- g_{D}(y, x)+(c_{0}-c_{1})\epsilon >  0\quad \text{when $n$ is sufficiently large.}
\end{align*}
Thus \text{for $x\in D\setminus \mathcal{N}_{D}$}, \begin{align*} h_{n}(x)-c_{1}g_{D}(y^{\star},x)&\geq \mathbb{E}^{x}\left[ h_{n}((X^{D})_{\tau_{V}})-c_{1}g_{D}(y^{\star},(X^{D})_{\tau_{V}})\right]\\ &= \mathbb{E}^{x}\left[\left( h_{n}(X_{\tau_{V}})-c_{1}g_{D}(y^{\star},X_{\tau_{V}})\right)\one_{\{\tau_{V}<\tau_{D}\}}\right] \geq0\end{align*} since  for $\mathcal{E}$-q.e. $x\in V$, $X_{\tau_{V}}\in\partial B(y^{\star},r)$ $\mathbb{P}^{x}$-a.s. on the set $\{\tau_{V}<\tau_{D}\}$. Letting $n$ tends to infinity, $c_{1}$ tends to $c_{0}$ and combining the continuity of Green function, we complete the proof.
\end{proof}

\section{Uniform domain}\label{s-ud}

In this section, we recall the definition of uniform domain and some properties of uniform domain that will be pertinent to our subsequent discussions. 

\begin{definition}[Uniform domain]\label{d-ud}
 Let $A\geq1$. A connected, non-empty, proper open set $U \subsetneq X$ is said to be an $A$-uniform domain if for every pair of points $x, y \in U$, there exists a curve $\gamma$ in $U$ from $x$ to $y$ such that its diameter $\diam(\gamma) \leq A d(x, y)$ and for all $z \in \gamma$,
$$
\delta_U(z) \geq A^{-1} \min (d(x, z), d(y, z)).
$$
Such a curve $\gamma$ is called an $A$-uniform curve.
\end{definition}
We can estimate the distance to the boundary for any point in a uniform curve.
\begin{lemma}\label{l-dist}
    Let $U$ be an $A$-uniform domain in $(\mathcal{X},d)$. If $x,y\in U$, then there exists an $A$-uniform curve $\gamma$ connecting $x$ and $y$ with $\delta_{U}(z)\geq(2A)^{-1}\min(\delta_{U}(x),\delta_{U}(y))$ for all $z\in\gamma$.
\end{lemma}
\begin{proof}
    Denote $s=\min(\delta_{U}(x),\delta_{U}(y))$. Let $\gamma$ be an $A$-uniform curve from $x$ to $y$ given by the uniform condition and let $z\in\gamma$. If $d(x, z)\leq s/2$, then $\delta_{U}(z)\geq\delta_{U}(x)-d(x, z)\geq s/2$ by triangle inequality. The same lower bound holds if $d(y, z)\leq s/2$. If $\min(d(x, z),d(y, z))\geq s/2$, then $\delta_{U}(z)\geq (2A)^{-1}s$ by definition.
\end{proof}
{The following two lemmas allow us to pick a point in the uniform domain that is away from the boundary.}
\begin{lemma}\label{l-pick}
    Let $U$ be an $A$-uniform domain in $(\mathcal{X},d)$. For every $x\in{U}$, $r>0$, if $B_{U}(x, r)\neq U$, then for any $s\in(0,r)$ there exist $x_{s}\in B_{U}(x, r)$ with 
    \begin{equation*}
        d(x, x_{s})=s\quad\text{and}\quad \delta_{U}(x_{s})\geq A^{-1}\min(s, r-s).
    \end{equation*}
\end{lemma}
\begin{proof}
    Since $B_{U}(x, r)\neq U$, there exists some $y\in U\setminus B_{U}(x, r)$. Let $\gamma$ be an $A$-uniform curve given by the uniform condition. As $d(x, y)\geq r>s$ and the distance function $d(x,\cdot)$ is continuous, we may take a point $x_{s}\in\gamma$ such that $d(x, x_{s})=s$ and consequently $d(x_{s},y)\geq d(x, y)-d(x, x_{s})>r-s$ by triangle inequality. By the definition of uniform curve, $\delta_{U}(x_{s})\geq A^{-1}\min(d(x, x_{s}),d(x_{s},y))\geq A^{-1}\min(s, r-s)$.
\end{proof}
\begin{lemma}[See {\cite[Proposition 3.20]{GSC11}}]\label{l-pick2}
    Let $U$ be an $A$-uniform domain in $(\mathcal{X},d)$. For any $\xi\in\partial U$, $r>0$ such that $B_{U}(\xi, r)\neq U$, there exists $\xi_{r}\in U\cap \partial B(\xi, r)$ such that $\delta_{U}(\xi_{r})\geq (2A)^{-1}r$.
\end{lemma}
\begin{proof}
    Since $B_{U}(\xi, r)\neq U$ implies there exist some $y\in U\setminus B_{U}(\xi, r)$. Let $x\in B_{U}(\xi,r/2)$ and  $\gamma$ be an $A$-uniform curve in $U$ connecting $x$ and $y$ and let $\xi_{r}\in \gamma\cap \partial B(\xi, r)$. By Lemma \ref{l-dist}, $\delta_{U}(\xi_{r})\geq A^{-1}\min(r/2, r)=(2A)^{-1}r$. 
\end{proof}
{The following lemma estimates the range of a uniform curve.}
\begin{lemma}\label{l-curv-b}
    Let $\xi\in\partial U$, $r>0$, $U\cap \partial B_{U}(\xi, r)\neq\emptyset$. Let $y_{1},y_{2}\in U\cap \partial B_{U}(\xi, r)$ and $\gamma$ be an $A$-uniform curve from $y_{1}$ to $y_{2}$ in $U$, then $\gamma\cap B_{U}(\xi, (A+1)^{-1}r)=\emptyset$ and $\gamma\subset\overline{B_{U}(\xi,(1+2A)r)}$.
\end{lemma}
\begin{proof}
    Let $z\in\gamma$. If $d(y_{1},z)\leq A(A+1)^{-1}r$, then by triangle inequality $d(z,\xi)\geq d(y_{1},\xi)-d(y_{1},z)\geq r-A(A+1)^{-1}r=(A+1)^{-1}r$. The same lower bound holds for $d(y_{2},z)\leq A(A+1)^{-1}r$. If $\min(d(y_{1},z),d(y_{2},z))> A(A+1)^{-1}r$, by the definition of uniform curve, \[\delta_{U}(z)\geq A^{-1}\min(d(y_{1},z),d(y_{2},z))>(A+1)^{-1}r.\] In particular, $z\notin B_{U}(\xi,(A+1)^{-1}r)$. For the second assertion, we notice that, by the definition of uniform curve, \[d(z,\xi)\leq r+\min(d(y_{1},z),d(y_{2},z))\leq r+\diam(\gamma)\leq (1+2A)r.\]
\end{proof}
\section{Proof of main result}\label{s-proof}
\subsection{Representation formula for harmonic functions}
The following Proposition is used to represent a non-negative harmonic function by its boundary values and Green function, which will help us to reduce BHP to the estimates of Green functions. The proof is the same as \cite[Proposition 4.3]{Lie15} except that the Green functions here are provided by Theorem \ref{t-g-ex} instead of heat kernel estimates.

\begin{proposition}[Representation formula for harmonic functions]\label{p-blyg}
    Let $U$ be a connected open subset of $\mathcal{X}$ such that $(\mathcal{E},\mathcal{F}^{0}(U))$ is transient. Let $0<r<r_{1}<r_{2}<\infty$ and $\xi\in\partial U$ such that $B_{U}(\xi, 5r_{2})\subsetneq \mathcal{X}$. Let $f\in\mathcal{F}^{0}_{\mathrm{loc}}(U,B_{U}(\xi, r_{2}))$ and suppose that $f$ is non-negative $\mathcal{E}$-harmonic in $B_{U}(\xi, r_{1})$. Then $f$ admits an $m$-version $\widetilde{f}$ which is continuous on $B_{U}(\xi, r)$ and there exists a Radon measure $\nu_{f}$ on $B_{U}(\xi, r_{1})$ supported in $U\cap\partial B_{U}(\xi, r)$ and satisfies \begin{equation*}
        \widetilde{f}(x)=\int_{U\cap\partial B_{U}(\xi, r)}g_{B_{U}(\xi, r_{1})}(x, y)\dif \nu_{f}(y),\quad\forall x\in B_{U}(\xi, r).
    \end{equation*}
\end{proposition}
\begin{proof}
     Since $f\in\mathcal{F}^{0}_{\mathrm{loc}}(U,B_{U}(\xi,r_{2}))$, and $B_{U}(\xi,r_{1})$ is relatively compact in $\overline{U}$ with \[d(B_{U}(\xi,r_{1}),U\setminus B_{U}(\xi,r_{2}))\geq r_{2}-r_{1}>0,\] we can find $f^{\#}\in \mathcal{F}^{0}(U)$, $f^{\#}=f\ m$-a.e. on $B_{U}(\xi,r_{1})$. Since $f$ is non-negative $\mathcal{E}$-harmonic on $B_{U}(\xi, r_{1})$, we may let $f^{\#}\geq0$ $m$-a.e. on $U$ and $\mathcal{E}$-harmonic on $B_{U}(\xi, r_{1})$. The rest of the proof is divided into several steps.
\begin{enumerate}[label=\textit{Step \arabic*.},align=left,leftmargin=*,topsep=5pt,parsep=0pt,itemsep=2pt]
        \item \emph{To find $g\in(\mathcal{F}^{0}(B_{U}(\xi, r_{1})))_{e}$ with ${g}={f^{\#}}$ $\mathcal{E}$-q.e. on $B_{U}(\xi, r)\cup(U\cap \partial B_{U}(\xi, r))$.}\footnote{Recalling that we represent every function in $\mathcal{F}_{e}$ by its $\mathcal{E}$-quasi-continuous version (see Remark \ref{r-quasicont}), the statement ${g}={f^{\#}}$ $\mathcal{E}$-q.e. makes sense.}

        We observe by \cite{GH14} that $P_{t}^{(B_{U}(\xi, r_{1}))}f^{\#}\leq f^{\#}\ m$-a.e. on $B_{U}(\xi, r_{1})$. Here $\{P_{t}^{(B_{U}(\xi, r_{1}))}\}_{t>0}$ is the heat semigroup corresponding to the killed Dirichlet form $(\mathcal{E},\mathcal{F}^{0}(B_{U}(\xi, r_{1})))$. Define a closed convex subset of $L^{2}(B_{U}(\xi, r_{1}),m)$ by \[K=\left\{u\in L^{2}(B_{U}(\xi, r_{1}),m): u\leq f^{\#} m\text{-a.e. on }B_{U}(\xi, r_{1})\right\}\] Note that the operation $L^{2}(B_{U}(\xi, r_{1}),m)\ni u\mapsto u\wedge f^{\#}$ is the projection map of $L^{2}(B_{U}(\xi, r_{1}),m)$ on $K$.  Then by \cite[Corollary 2.4]{Ouh96} we obtain that \[u\wedge f^{\#}\in \mathcal{F}^{0}(B_{U}(\xi, r_{1}))\text{ and }\mathcal{E}(u\wedge f^{\#}, u\wedge f^{\#})\leq \mathcal{E}(u, u),\;\forall u\in \mathcal{F}^{0}(B_{U}(\xi, r_{1})).\]For any $u\in (\mathcal{F}^{0}(B_{U}(\xi, r_{1})))_{e}$, we may find an $\mathcal{E}$-Cauchy sequence of $u_{n}$ in $\mathcal{F}^{0}(B_{U}(\xi, r_{1}))$ that converges to $u$ $m$-a.e.. Therefore, the above property can be extended further to \begin{equation}\label{e-blyg1}
            u\wedge f^{\#}\in (\mathcal{F}^{0}(B_{U}(\xi, r_{1})))_{e}\text{ and }\mathcal{E}(u\wedge f^{\#}, u\wedge f^{\#})\leq \mathcal{E}(u,u),\;\forall u\in (\mathcal{F}^{0}(B_{U}(\xi, r_{1})))_{e}.
        \end{equation}

        Define a closed convex subset of $(\mathcal{F}^{0}(B_{U}(\xi, r_{1})))_{e} $ by \[K_{1}=\left\{g\in (\mathcal{F}^{0}(B_{U}(\xi, r_{1})))_{e}: {g}\geq {f^{\#}}\text{ $\mathcal{E}$-q.e. on }B_{U}(\xi, r)\cup(U\cap \partial B_{U}(\xi, r))\right\}.\]  By the same proof of \cite[Lemma 2.1.1]{FOT11} we conclude that there is a unique $g\in K_{1}$ such that $\mathcal{E}(g, g)=\inf_{u\in K_{1}} \mathcal{E}(u, u)$. Applying \eqref{e-blyg1} to $u=g$ we conclude $u\wedge f^{\#}\in K$ also minimizes $\mathcal{E}(u, u)$ in $K_{1}$,  which implies ${g}={u\wedge f^{\#}}\leq {f^{\#}}$ by uniqueness of $g$. Therefore ${g}={f^{\#}}$ $\mathcal{E}$-q.e. on $B_{U}(\xi, r)\cup(U\cap \partial B_{U}(\xi, r))$.

\item \emph{To show that $\mathcal{E}(g,v)\geq0$ for all $v\in (\mathcal{F}^{0}(B_{U}(\xi, r_{1})))_{e}$ with $v\geq0$ $\mathcal{E}$-q.e. on $B_{U}(\xi,r)\cup(U\cap \partial B_{U}(\xi,r))$.}

Assume $v\in (\mathcal{F}^{0}(B_{U}(\xi,r_{1})))_{e}$ with $v\geq0$ $\mathcal{E}$-q.e. on $B_{U}(\xi,r)\cup(U\cap \partial B_{U}(\xi,r))$. By definition, $g+tv\in K_{1}$ for all $t>0$ and consequently $\mathcal{E}(g+tv,g+tv)\geq \mathcal{E}(g,g)$ for all $t>0$, which implies $\mathcal{E}(g,v)\geq0$.

        \item \emph{To show that $\mathcal{E}(g,v)=0$ for all $v\in (\mathcal{F}^{0}(B_{U}(\xi,r_{1})))_{e}$ with $v=0$ $\mathcal{E}$-q.e. on $U\cap \partial B_{U}(\xi,r)$.}
        
        By \cite[Theorem 2.3.3]{FOT11}, it is sufficient to show that $\mathcal{E}(g,v)=0$ for all $v$ in $\mathcal{F}^{0}(B_{U}(\xi,r_{1}))\cap C_{c}(B_{U}(\xi,r_{1}))$ with $\supp(v)\subset B_{U}(\xi,r_{1})\setminus (U\cap \partial B_{U}(\xi,r))$. As $\supp(v)$ is compact, $\dist(\supp(v), U\cap \partial B_{U}(\xi,r)):=\epsilon>0$, thus $\supp(v)\subset B_{U}(\xi,r-\epsilon)\cup \left(B_{U}(\xi,r_{1})\setminus  B_{U}(\xi,r+\epsilon)\right)$. Therefore by \cite[Proposition A.4]{GH14}, $v\one_{B_{U}(\xi,r)}$ and $v\one_{B_{U}(\xi,r_{1})\setminus  B_{U}(\xi,r+\epsilon)}$ both belong to $\mathcal{F}^{0}(B_{U}(\xi,r_{1}))\cap C_{c}(B_{U}(\xi,r_{1}))$. By the $\mathcal{E}$-harmonicity of $f$ on $B_{U}(\xi,r)$, we have $\mathcal{E}(g,v\one_{B_{U}(\xi,r)})=\mathcal{E}(f,v\one_{B_{U}(\xi,r)})=0$. By Step 2, $\mathcal{E}(g,v\one_{B_{U}(\xi,r_{1})\setminus  B_{U}(\xi,r+\epsilon)})=0$. As a result, $\mathcal{E}(g,v)=\mathcal{E}(g,v\one_{B_{U}(\xi,r)})+\mathcal{E}(g,v\one_{B_{U}(\xi,r_{1})\setminus  B_{U}(\xi,r+\epsilon)})=0$.

        \item \emph{To show that $\mathcal{E}(g,v)\geq0$ for all $v\in (\mathcal{F}^{0}(B_{U}(\xi,r_{1})))_{e}$ with $v\geq0$ $\mathcal{E}$-q.e. on $U\cap \partial B_{U}(\xi,r)$.} 

        Write $v_{+}=\max(v,0)$ which belongs to $(\mathcal{F}^{0}(B_{U}(\xi,r_{1})))_{e}$ by Markovian property of Dirichlet form. Then $v-v_{+}=0$ $\mathcal{E}$-q.e. on $U\cap \partial B_{U}(\xi,r)$. Thus $\mathcal{E}(g,v-v_{+})=0$ by Step 3 and $\mathcal{E}(g,v_{+})\geq0$ by Step 2. Hence $\mathcal{E}(g,v)=\mathcal{E}(g,v-v_{+})+\mathcal{E}(g,v_{+})\geq0$.

        \item \emph{To find such $\nu_{f}$.}

        By the $0$-order version of \cite[Lemma 2.2.6]{FOT11}, $g$ is a $0$-order potential for some positive Radon measure $\nu_{f}$ on $B_{U}(\xi,r_{1})$ of finite $0$-order energy integral for $(\mathcal{E},(\mathcal{F}^{0}(B_{U}(\xi,r_{1})))_{e})$ such that $\nu_{f}(B_{U}(\xi,r_{1})\setminus(U\cap\partial B_{U}(\xi,r))=0$ and \begin{equation}\label{e-blyg2}
    \mathcal{E}(g,v)=\int_{U\cap\partial B_{U}(\xi,r)}{v}(x)\dif\nu_{f}(x),\quad\forall v\in (\mathcal{F}^{0}(B_{U}(\xi,r_{1})))_{e}.
\end{equation}
Denote $h$ as a reference function of the transient Dirichlet space $(\mathcal{E},\mathcal{F}^{0}(B_{U}(\xi,r_{1})))$, i.e.\[h(x)>0 \text{ for all }x\in B_{U}(\xi,r_{1}) \text{ and }\int_{B_{U}(\xi,r_{1})}h\cdot G^{B_{U}(\xi,r_{1})}h\dif m<\infty.\]Here $G^{B_{U}(\xi,r_{1})}$ is the Green operator defined by \[G^{B_{U}(\xi,r_{1})}h(x)=\mathbb{E}^{x}\left[\int_{0}^{\tau_{B_{U}(\xi,r_{1})}}h(X_{s})\dif s\right], \quad x\in B_{U}(\xi,r_{1}).\] Given $z\in B_{U}(\xi,r_{1})$ and $r>0$ small enough such that $B(z,r)\subset B_{U}(\xi,r_{1})$, let
\[v_{z,r}(x)=\int_{B_{U}(\xi,r_{1})}g_{B_{U}(\xi,r_{1})}(x,y)h(y)\one_{B(z,r)}(y)\dif m(y), \quad x\in B_{U}(\xi,r_{1}).\]
By Theorem \ref{t-g-ex}, \[v_{z,r}(x)=G^{B_{U}(\xi,r_{1})}(\one_{B(z,r)}\cdot h)(x) \text{ for every }x\in B_{U}(\xi,r_{1})\setminus \mathcal{N}_{B_{U}(\xi,r_{1})}.\]
where $\mathcal{N}_{B_{U}(\xi,r_{1})}$ is a Borel properly exceptional set in $B_{U}(\xi,r_{1})$. Moreover, $v_{z,r}\in (\mathcal{F}^{0}(B_{U}(\xi,r_{1})))_{e}$ by \cite[Theorem 1.5.4(ii)]{FOT11}. Hence we can apply \eqref{e-blyg2} to $v_{z,r}$ and use Fubini's Theorem, 
\begin{align}
    &\int_{B_{U}(\xi,r_{1})}g(y)\one_{B(z,r)}(y)h(y)\dif m(y)\\=\ &\mathcal{E}(g,G^{B_{U}(\xi,r_{1})}(\one_{B(z,r)}\cdot h))\\=\ &\mathcal{E}(g,v_{z,r})=\int_{U\cap\partial B_{U}(\xi,r)}{v_{z,r}}(x)\dif\nu_{f}(x)\\ =\ & \int_{B_{U}(\xi,r_{1})}\left(\int_{U\cap\partial B_{U}(\xi,r)}g_{B_{U}(\xi,r_{1})}(x,y)\dif\nu_{f}(x)\right)\one_{B(z,r)}(y)h(y)\dif m(y).
\end{align}
Dividing both sides by $m(B(z,r))$ and letting $r\rightarrow0^{+}$, we can observe from Lebesgue's differentiation Theorem that \[g(z)=\int_{U\cap\partial B_{U}(\xi,r)}g_{B_{U}(\xi,r_{1})}(z,y)\dif\nu_{f}(y) \text{ for }m\text{-a.e. }z\in B_{U}(\xi,r_{1}).\]
Since $g=f$ $m$-a.e. on $B_{U}(\xi,r)$,  we have \[f(z)=\int_{U\cap\partial B_{U}(\xi,r)}g_{B_{U}(\xi,r_{1})}(z,y)\dif\nu_{f}(y) \text{ for }m\text{-a.e. }z\in B_{U}(\xi,r).\]
Notice that for any $z_{0}\in B_{U}(\xi,r)$, $\dist(z_{0},U\cap\partial B_{U}(\xi,r)):=\epsilon^{\prime}>0$. By the joint continuity of $g_{B_{U}(\xi,r_{1})}$ in $(B_{U}(\xi,r_{1})\times B_{U}(\xi,r_{1}))\setminus B_{U}(\xi,r_{1})_{\mathrm{diag}}$, $g_{B_{U}(\xi,r_{1})}$ is uniformly bounded on the set $B(z_{0},\epsilon^{\prime}/2)\times \left(U\cap\partial B_{U}(\xi,r)\right)$. Apply the dominated convergence theorem and we conclude that the function \[\widetilde{f}(z):= \int_{U\cap\partial B_{U}(\xi,r)}g_{B_{U}(\xi,r_{1})}(z,y)\dif\nu_{f}(y),\ z\in B_{U}(\xi,r),\]
is continuous on $B_{U}(\xi,r)$ and is an $m$-version of $f$.
    \end{enumerate}
\end{proof}

\subsection{Estimates on Green functions}

For the remainder of this section, we assume that $A\geq1$ and fix an $A$-uniform domain $U\subsetneq\mathcal{X}$. We notice that for any open subset $V\subset \mathcal{X}$,  if $\diam(V)<\diam(\mathcal{X})$, its complement $\mathcal{X}\setminus V$ is not $\mathcal{E}$-polar. Thus the regular Green function $g_{V}(\cdot,\cdot)$ on $(V\times V)\setminus V_{\mathrm{diag}}$ exists by Theorem \ref{t-g-ex}. We will also assume that $\hyperlink{RBC(K)}{\mathrm{RBC}(K)}$ and $\hyperlink{EHI}{\mathrm{EHI}(C_{\mathrm{H}}, \delta_{\mathrm{H}})}$ holds. We will use $A_{j}$ to denote constants which only depend on the constant $A$ in the uniform domain condition and $K$ in the relative ball connected condition, and will use $C_{j}$ to denote constant depending on $A,C_{\mathrm{H}},\delta_{\mathrm{H}},K$ and the doubling constants.

In the process of proving BHP, the concept of \emph{capacitary width} is frequently invoked. Compared to the previous literature \cite{Aik01,Lie15,BM19}, it is necessary to extend the notation to annuli with greater width, so that Theorem \ref{t-g-cp} can be applied.

\begin{definition}[Capacitary width]
  For an open set $V \subset \mathcal{X}$, $N\in[2,\infty)$ and $\eta \in(0,1)$, define the capacitary width $w_{N,\eta}(V)$ by
$$
w_{N,\eta}(V)=\inf \left\{r>0: \frac{\Capacity_{B(x, N r)}(\overline{B(x, r)} \setminus V)}{\Capacity_{B(x, N r)}(\overline{B(x, r)})} \geq \eta,\ \forall x \in V\right\} .
$$
\end{definition}
\begin{definition}[Harmonic measure]
Let $\Omega \subset \mathcal{X}$ be open and relatively compact in $\mathcal{X}$. Since the process $\left\{X_t\right\}_{t\geq0}$ has continuous path, $X_{\tau_{\Omega}} \in \partial \Omega$ $\mathbb{P}^{x}$-a.s. on the set $\{\tau_{\Omega}<\infty\}$ for all $x\in \mathcal{X}$. We define the \emph{harmonic measure} $\omega(x, \cdot, \Omega)$ on $\partial \Omega$ by setting
$$
\omega(x, F, \Omega):=\mathbb{P}^x\left(X_{\tau_{\Omega}} \in F,\tau_{\Omega}<\infty\right) \text { for Borel subset } F \subset \partial \Omega .
$$
\end{definition}
We record some properties of harmonic measures.
\begin{lemma}\label{l-hm-rh}
    Let $\Omega \subset \mathcal{X}$ be open and relatively compact in $\mathcal{X}$, $F \subset \partial \Omega$. 
    \begin{enumerate}[label=\textup{(\roman*)},align=left,leftmargin=*,topsep=5pt,parsep=0pt,itemsep=2pt]
        \item\label{lb.rh1}  \textup{(Domain monotonicity)} If $A\subset B$ are open subsets of $\mathcal{X}$, then\begin{equation}\label{e-hm-dm}
            \omega(x,\Omega\cap \partial B,\Omega\cap B)\leq \omega(x,\Omega\cap \partial A,\Omega\cap A),\quad\forall x\in\Omega\cap A.
        \end{equation}
        \item\label{lb.rh2}  \textup{(Mean-value property)} If $\mathcal{X}\setminus \Omega$ is not $\mathcal{E}$-polar. Then for any $V\subset \Omega$ with $\overline{V}\cap F=\emptyset$, the function $\omega(\cdot, F, \Omega)$ is regular harmonic in $V$ with respect to $X^{\Omega}$. Moreover, the following mean-value property holds\begin{equation}
        \label{e-hm-mp}
        \omega(x,F,\Omega)=\int_{\partial V\cap\Omega}\ \omega(z,F,\Omega
    )\omega(x,\dif z,V),\quad\text{for $\mathcal{E}$-q.e. } x\in V.
    \end{equation}
    \end{enumerate}
\end{lemma}
\begin{proof}
\begin{enumerate}[label=\textup{(\roman*)},align=left,leftmargin=*,topsep=5pt,parsep=0pt,itemsep=2pt]
    \item[\ref{lb.rh1}] Fix $x\in\Omega\cap A$. We first notice that $A\subset B$ implies $\tau_{\Omega\cap A}\leq \tau_{\Omega\cap B}$ $\mathbb{P}^{x}$-a.s. If $X_{\tau_{\Omega\cap B}}\in \Omega\cap\partial B\subset\Omega$, then by definition of $\tau_{\Omega}$ we know that $\tau_{\Omega\cap A}\leq\tau_{\Omega\cap B}<\tau_{\Omega}$. Thus \[X_{\tau_{\Omega\cap A}}\in \Omega\cap\partial (\Omega\cap A)\subset \Omega\cap\partial A.\]Hence \[\{X_{\tau_{\Omega\cap B}}\in \Omega\cap\partial B\}\subset\{X_{\tau_{\Omega\cap A}}\in \Omega\cap\partial A\},\quad \mathbb{P}^{x}\text{-a.s.}\]
    which implies \eqref{e-hm-dm} by the monotonicity of probability.
    \item[\ref{lb.rh2}] Since $\left\{X_t\right\}_{t\geq0}$ has continuous path, $X_{\tau_{V}} \in \partial V\subset \overline{V}\subset F^{c}$ $\mathbb{P}^{x}$-a.s.. Thus  $\one_{F}(X_{\tau_{V}})=0$ $\mathbb{P}^{x}$-a.s. for $\mathcal{E}$-q.e. $x\in V$. By the strong Markov property and the fact that $\{X_{\tau_{V}}\in\Omega\cap \partial V\}= \{\tau_{V}<\tau_{\Omega}\}$ $\mathbb{P}^{x}$-a.s. and $X^{\Omega}_{t}=X_{t}$ for $t<\tau_{\Omega}$, we have\begin{align*}
        \omega(x,F,\Omega)&=\mathbb{E}^{x}\left[\one_{F}(X_{\tau_{\Omega}})\right]\\ 
        &= \mathbb{E}^{x}\left[\one_{F}(X_{\tau_{\Omega}})\one_{\{\tau_{V}<\tau_{\Omega}\}}\right]+\mathbb{E}^{x}\left[\one_{F}(X_{\tau_{\Omega}})\one_{\{\tau_{V}=\tau_{\Omega}\}}\right]\\ &= \mathbb{E}^{x}\left[\one_{F}(X_{\tau_{\Omega}})\one_{\{\tau_{V}<\tau_{\Omega}\}}\right]+\mathbb{E}^{x}\left[\one_{F}(X_{\tau_{V}})\one_{\{\tau_{V}=\tau_{\Omega}\}}\right]\\&=\mathbb{E}^{x}\left[\mathbb{E}^{x}\left[\one_{F}(X_{\tau_{\Omega}})\one_{\{\tau_{V}<\tau_{\Omega}\}}\,|\ \mathscr{F}_{\tau_{V}}\right]\right]\\ &= \mathbb{E}^{x}\left[\mathbb{E}^{X_{\tau_{V}}}\left[\one_{F}(X_{\tau_{\Omega}})\right]\one_{\{\tau_{V}<\tau_{\Omega}\}}\right]\ \text{(since }\{\tau_{V}<\tau_{\Omega}\}\in \mathscr{F}_{\tau_{V}}\text{)}\\&= \mathbb{E}^{x}\left[\mathbb{E}^{(X^{\Omega})_{\tau_{V}}}\left[\one_{F}(X_{\tau_{\Omega}})\right]\right]= \mathbb{E}^{x}\left[\omega\left((X^{\Omega})_{\tau_{V}},F,\Omega\right)\right].
    \end{align*} 
    which means that $\omega(\cdot, F, \Omega)$ is regular harmonic in $V$ with respect to $X^{\Omega}$. As we have seen in the above calculation, \begin{align*}
        \omega(x,F,\Omega)&=\mathbb{E}^{x}\left[\omega(X_{\tau_{V}},F,\Omega)\one_{\{X_{\tau_{V}}\in\partial V\cap\Omega\}}\right]\\&=\int_{\partial V\cap\Omega}\ \omega(z,F,\Omega
    )\omega(x,\dif z,V).
    \end{align*}
\end{enumerate}
\end{proof}
The following lemma is used to estimate Green functions by harmonic measures.\begin{lemma}\label{l-g-hmk}
    Let $D$ be a bounded, non-empty open subset of $U$.  Let $g_{D}(x,y)$ be the regular Green function on $D$. For any open subset  $\Omega\subset D$ such that $x\notin \overline{\Omega}$, we have \begin{equation}
        g_{D}(x, y)\leq \omega(y,U\cap \partial\Omega,\Omega)\sup_{z\in U\cap \partial\Omega}g_{D}(x, z),\quad\text{for $\mathcal{E}$-q.e. $y\in\Omega$}.
    \end{equation}
\end{lemma}
\begin{proof}
    By the regular harmonicity of $g_{D}(x,\cdot)$ in $\Omega$ with respect to $X^{D}$, we have for $\mathcal{E}$-q.e. $y\in\Omega$, \begin{align*}
        g_{D}(x, y)&=\mathbb{E}^{y}\left[g_{D}(x,(X^{D})_{\tau_{\Omega}})\right]\\ &=\mathbb{E}^{y}\left[g_{D}(x,X_{\tau_{\Omega}})\one_{\{\tau_{\Omega}<\tau_{D}\}}\right]\\ &= \mathbb{E}^{y}\left[g_{D}(x,X_{\tau_{\Omega}})\one_{\{X_{\tau_{\Omega}}\in\partial\Omega\cap D\}}\right]
        \\ &=\int_{\partial \Omega\cap D}g_{D}(x,z)\omega(y,\dif z,\Omega)\\ &\leq \omega(y,D\cap \partial\Omega,\Omega)\sup_{z\in D\cap \partial\Omega}g_{D}(x,z)\\& \leq \omega(y,U\cap \partial\Omega,\Omega)\sup_{z\in U\cap \partial\Omega}g_{D}(x,z).
    \end{align*}
\end{proof}

Next lemma relates capacitary width and harmonic measure. {The proof is same as in \cite[Lemma 4.13]{GSC11}. The difference is that we use Theorem \ref{t-g-cp} here instead of \cite[Lemma 4.8, (4,9)]{GSC11}, which used heat kernel, to obtain \cite[(4.10)]{GSC11}.}

\begin{lemma}\label{l-hm-cw}
There exists $a_1 \in$ $(0,1)$ such that for any non-empty open set $V \subset \mathcal{X}$ and for all $x \in \mathcal{X}, r>0$, $\eta\in(0,1)$, $N\geq N_{K}=10K$,
$$
\omega(x, V \cap \partial B(x, r), V \cap B(x, r)) \leq \exp \left(2-\frac{a_1 r}{w_{N,\eta}(V)}\right)
$$
\end{lemma}
\begin{proof}
For any $\kappa\in(0,1)$, we can find $s$ with $w_{N,\eta}(V)\leq s<w_{N,\eta}(V)+\kappa$ such that \[\frac{\Capacity_{B(y, N s)}(\overline{B(y, s)} \setminus V)}{\Capacity_{B(y, N s)}(\overline{B(y, s)})} \geq \eta,\quad\text{for all }y\in V.\] Fix $y\in V$ and let $E=\overline{B(y, s)} \setminus V$. Let $\nu_{E}$ be the equilibrium measure of $E$ in $\Omega:=B(y, N s)$. We claim that there exists $\epsilon>0$ such that \begin{equation}\label{e-hmcw-cl}\int_{E} g_{\Omega}(z, \cdot)\dif \nu_{E}\geq\epsilon\eta \quad\text{for all }z\in B(y,s).\end{equation}In fact, let $F=\overline{B(y, s)}$ and $\nu_{F}$ be the equilibrium measure of $F$ in $\Omega$. For any $z\in\partial B(y,3s/2)$ and any $\zeta\in B(y,s)$, we have $\{y,z,\zeta\}\subset B(y,2s)$, $\min(d(z,\zeta),d(z,y))\geq s/2$ and $B(y,2(3+2K)s)\subset \Omega$, we conclude by Theorem \ref{t-g-cp} and the continuity of Green functions that \[\oldconstant{c2}^{-1}g_{\Omega}(z,y)\leq g_{\Omega}(z,\zeta)\leq \oldconstant{c2}g_{\Omega}(z,y) \text{ for all } \zeta\in F \text{ and all }z\in\partial B(y,3s/2) .\] Hence \begin{equation}\int_{E} g_{\Omega}(z,\cdot)\dif \nu_{E}\in\left[\oldconstant{c2}^{-1}g_{\Omega}(z,y)\nu_{E}(E),\oldconstant{c2}g_{\Omega}(z,y)\nu_{E}(E)\right]\text{ for all }z\in\partial B(y,3s/2)\end{equation} and \begin{equation}\label{e-hmcw1}\int_{F} g_{\Omega}(z,\cdot)\dif \nu_{F}\in\left[\oldconstant{c2}^{-1}g_{\Omega}(z,y)\nu_{F}(F),\oldconstant{c2}g_{\Omega}(z,y)\nu_{F}(F)\right]\text{ for all }z\in\partial B(y,3s/2)\end{equation}Moreover, since $\nu_{F}(F)=\Capacity_{\Omega}(\overline{B(y,s)})\leq \Capacity_{\Omega}(B(y,3s/2))\leq \left(\min_{z^{\prime}\in\partial B(y,3s/2)}g_{\Omega}(z^{\prime},y)\right)^{-1}$, we see that 
\begin{align*}
g_{\Omega}(z,y)\nu_{F}(F)&\leq \left(\max_{z^{\prime}\in\partial B(y,3s/2)}g_{\Omega}(z^{\prime},y)\right)\nu_{F}(F)\\&\leq \oldconstant{c1}\left(\min_{z^{\prime}\in\partial B(y,3s/2)}g_{\Omega}(z^{\prime},y)\right)\nu_{F}(F)\leq  \oldconstant{c1}.
\end{align*}
On the other hand, since \begin{align}
\nu_{F}(F)=\Capacity_{\Omega}(\overline{B(y,s)})&\geq\Capacity_{\Omega}({B(y,s)})\\
&\geq \left(\min_{z^{\prime}\in\partial B(y,s)}g_{\Omega}(z^{\prime},y)\right)^{-1}\geq \oldconstant{c2}^{-1} \left(\min_{z^{\prime}\in\partial B(y,3s/2)}g_{\Omega}(z^{\prime},y)\right)^{-1},
\end{align} we have 
\[g_{\Omega}(z,y)\nu_{F}(F)\geq \left(\min_{z^{\prime}\in\partial B(y,3s/2)}g_{\Omega}(z^{\prime},y)\right)\nu_{F}(F)\geq  \oldconstant{c2}^{-1}.\]
Thus \eqref{e-hmcw1} can be refined as \[\oldconstant{c2}^{-2}\leq\int_{F} g_{\Omega}(z,\cdot)\dif \nu_{F} \leq \oldconstant{c1}\oldconstant{c2}\text{ for all }z\in\partial B(y,3s/2). \]
Hence for $z\in\partial B(y,3s/2)$, \begin{align*}\int_{E} g_{\Omega}(z, \cdot)\dif \nu_{E}&\geq\left(\int_{E} g_{\Omega}(z, \cdot)\dif \nu_{E}\right)\left(\oldconstant{c2}^{2}\int_{F} g_{\Omega}(z, \cdot)\dif \nu_{F}\right)^{-1}\\ &\geq \frac{\oldconstant{c2}^{-1}g_{\Omega}(z,y)\nu_{E}(E)}{\oldconstant{c2}^{3}g_{\Omega}(z,y)\nu_{F}(F)}=\frac{1}{\oldconstant{c2}^{4}}\frac{\Capacity_{B(y, N s)}(\overline{B(y, s)} \setminus V)}{\Capacity_{B(y, N s)}(\overline{B(y, s)})} \geq \oldconstant{c2}^{-4}\eta.\end{align*}
We know that $ \int_{E} g_{\Omega}(z, \cdot)\dif \nu_{E}=\mathbb{P}^{z}(\tau_{E^{c}}<\tau_{\Omega})$ for $\mathcal{E}$-q.e. $z\in \Omega$. So for $\mathcal{E}$-q.e. $z\in B(y,s)$, by strong Markov property,\begin{align*}
\mathbb{P}^{z}(\tau_{E^{c}}<\tau_{\Omega})&=\mathbb{E}^{z}\left[\one_{\{\tau_{E^{c}}<\tau_{\Omega}\}}\right]\\&= \mathbb{E}^{z}\left[\mathbb{E}^{X_{\tau_{B(y,3s/2)}}}\left[\one_{\{\tau_{E^{c}}<\tau_{\Omega}\}}\right]\right]\geq \oldconstant{c2}^{-4}\eta.
\end{align*} Thus \eqref{e-hmcw-cl} holds by continuity of Green function. The rest is the same as \cite[Lemma 4.13]{GSC11}.
\end{proof}
The following lemma is an estimate of capacity width, which is an analogue of  \cite[(2.1)]{Aik01}, \cite[Lemma 5.2]{BM19} and \cite[Lemma 4.12]{GSC11}.
\begin{lemma}\label{l-capw}
    There exists $\eta=\eta(A,K,C_{\mathrm{H}},\delta_{\mathrm{H}})$ such that for $N=N_{K}$, and for any $0<r<(\max(13A,100N_{K}))^{-1}\diam(U)$, there holds \[w_{N,\eta}\left(\{x\in U:\delta_{U}(x)<r\}\right)\leq 6Ar.\]
\end{lemma}
\begin{proof}
    Denote $V_{r}=\{x\in U:\delta_{U}(x)<r\}$. For any $x\in V_{r}$, $B_{U}(x,6Ar)\neq U$ since $r<(13A)^{-1}\diam(U)$. By Lemma \ref{l-pick}, there exists $z\in B_{U}(x,6Ar)$ with the property that $d(x,z)=3Ar$ and $\delta_{U}(z)\geq (2A)^{-1}6Ar>2r.$ By the domain monotonicity of capacity, \[\Capacity_{B(x, 6NAr)}(\overline{B(x, 6Ar)} \setminus V_{r})\geq \Capacity_{B(x, 6NAr)}(\overline{B(z, r)})\geq \Capacity_{B(z, 6(N+1)Ar)}(\overline{B(z, r)}).\]
    By Theorem \ref{t-g-cp}, we know that     \begin{align*}
    \Capacity_{B(z, 6(N+1)Ar)}(\overline{B(z, r)})&\geq\Capacity_{B(z, 6(N+1)Ar)}({B(z, r)})\\ &\geq \oldconstant{c3}^{-1}  \Capacity_{B(z, 6(N+1)Ar)}({B(z, 4Ar)})\\ &\geq \oldconstant{c3}^{-1}  \Capacity_{B(x, 3(2N+3)Ar)}({B(x, Ar)})\\ &\geq \oldconstant{c3}^{-2}  \oldconstant{c3a}^{{-1}}\Capacity_{B(x, 6NAr)}({B(x, 7Ar)}) \\&\geq \oldconstant{c3}^{-2}\oldconstant{c3a}^{{-1}}\Capacity_{B(x, 6NAr)}(\overline{B(x, 6Ar)}).
    \end{align*}
 So for $\eta=\oldconstant{c3}^{-2}\oldconstant{c3a}^{{-1}}$, and any $r<(13A)^{-1}\diam(U)$, we have \[\frac{\Capacity_{B(x, 6NAr)}(\overline{B(x, 6Ar)} \setminus V_{r})}{\Capacity_{B(x, 6NAr)}(\overline{B(x, 6Ar)})}\geq \eta,\]
    which means $w_{N,\eta}(V_{r})\leq 6Ar$ by definition.
\end{proof}
The following lemma gives another type estimates on length of Harnack chain (compared to Proposition \ref{p-hc-p}). {This type of chain allows us to circumvent the singularities of the Green functions, and exhibits logarithmic growth rates (instead of polynomial growth as in Proposition \ref{p-hc-p}), which plays an important role in the box argument of Lemma \ref{l-hmkm}. Since the metric space under consideration is not necessarily geodesic and may lack non-trivial rectifiable curves, a novel argument is required to give this kind of estimate.}
\begin{lemma}\label{l-hc-log}
    Let $A_{1}\geq3$, $\xi\in\partial U$ and $x\in B_{U}(\xi,2r)$. If $\eta\in U$ satisfies $d(\xi,\eta)=A_{1}r$ and $\delta_{U}(\eta)\geq (2A)^{-1}A_{1}r$. Then there exists a point $\widetilde{\eta}$ with $d(\eta,\widetilde{\eta})=A^{-1}r$ such that for any $M>1$ there exists an $M$-Harnack chain from $\widetilde{\eta}$ to $x$ in $B_{U}(\xi,4A^{2}(A_{1}+2)r)\setminus \{\eta\}$ whose length is less than \[C\log\left(1+\frac{r}{\delta_{U}(x)}\right)+C,\]
    where $C=C(A,A_{1},M)$ is a constant depending only on $A$, $M$ and the doubling constant of $m$.
\end{lemma}
\begin{proof}
The proof is motivated by \cite[Lemma 2.23]{KM23a}. Let $\gamma:[0,1]\rightarrow U$ be a uniform curve from $\eta$ to $x$ with $\gamma(0)=\eta$ and $\gamma(1)=x$. By definition of uniform curve, $\diam(\gamma)\leq Ad(\eta,x)\leq A(A_{1}+2)r$. Define
        \[t_{0}=\sup\{t\in[0,1]: d(\eta,\gamma(t))<d(\gamma(t),x) \}.\]
        Then $t_{0}\in(0,1)$. By the continuity of distance function, we have 
        \begin{align}
            d(\eta,\gamma(t_{0}))&=d(\gamma(t_{0}),x)\\
            &= \left(d(\eta,\gamma(t_{0}))+d(\gamma(t_{0}),x)\right)/2\\
            &\geq d(\eta,x)/2\geq \left(d(\eta,\xi)-d(\xi,x)\right)/2\geq (A_{1}-2)r/2,
        \end{align} and \[d(\eta,\gamma(t))\geq d(\gamma(t),x) \text{ and }d(\eta,\gamma(t))\geq (A_{1}-2)r/2 \quad \text{for all }t_{0}\leq t\leq1.\]  Let $q=(3A)^{-2}(A_{1}+2)^{-1}<1$. Choose $\{t_{j}\}_{j\geq1}$ such that $t_{0}<t_{1}<\ldots<1$ and \[d(\gamma(t_{j}),x)=q^{j}d(\gamma(t_{0}),x)\in\left(\frac{1}{2}(A_{1}-2)q^{j}r, A(A_{1}+2)q^{j}r\right).\] By definition of uniform curve, we have \[\delta_{U}(\gamma(t_{j}))\geq A^{-1}\min(d(\eta,\gamma(t_{j})),d(\gamma(t_{j}),x))\geq (2A)^{-1}(A_{1}-2)q^{j}r,\quad j\geq0.\] 
        
Define $\widetilde{\gamma}$ to be an $A$-uniform curve from $\eta$ to $\gamma(t_{1})$ and define $\widetilde{\eta}$ as the last point of $\widetilde{\gamma}$ that leaves $B(\eta,A^{-1}r)$. By Proposition \ref{p-hc-p}, we can find an $M$-Harnack chain from $\eta$ to $\gamma(t_{1})$ centred in $\widetilde{\gamma}$, with radius $(4A)^{-2}(A_{1}-2)qr$ and length less than $C_{D}\left(1+80MA^{4}q^{-1}\right)^{\alpha}$. This chain is contained in $B(\xi, 4A^{2}(A_{1}+2)r)\setminus\{\eta\}$

For $j\geq1$, let $\gamma_{j}$ be an $A$-uniform curve from $\gamma(t_{j})$ to $\gamma(t_{j+1})$. Then
        \begin{align*}
            d(\gamma(t_{j}),\gamma(t_{j+1}))&\leq d(\gamma(t_{j}),x)+d(\gamma(t_{j+1}),x)\\&\leq (q^{j}+q^{j+1})d(\gamma(t_{0}),x) \leq 2A(A_{1}+2)q^{j}r,
        \end{align*}
        which implies $\diam(\gamma_{j})\leq 2A^{2}(A_{1}+2)q^{j}r \ \text{ and }$ \[ \gamma_{j}\subset B_{U}(\gamma(t_{j}),2A^{2}(A_{1}+2)q^{j}r)\subset B(\xi, 3A^{2}(A_{1}+2))\setminus B(\eta, (2A)^{-1}r)\]
        since $2A^{2}(A_{1}+2)q^{j}r+(2A)^{-1}r\leq 2A^{2}(A_{1}+2)qr+(2A)^{-1}r<r\leq d(\eta,\gamma(t_{j}))$. By Proposition \ref{p-hc-p}, we can find an $M$-Harnack chain from $\gamma(t_{j})$ to $\gamma(t_{j+1})$ centred in $\gamma_{j}$ with length less than $C_{D}(1+24MA^{4}q^{-1})^{\alpha}$. These Harnack chains are all contained in $B_{U}(\xi, 4A^{2}(A_{1}+2)r)\setminus\{\eta\}$. We stop this procedure when  $j$ is the smallest number such that $\gamma(t_{j})\in B(x,M^{-1}\delta_{U}(x))$. So $j$ would be less than \[1+\left(\log\frac{1}{q}\right)^{-1}\log\left(1+\frac{r}{6AM\delta_{U}(x)}\right).\] By re-ordering these balls in an obvious way, we can construct such an $M$-Harnack chain. Adding all the lengths of Harnack chains, we have the desired bound.
\end{proof}
Define \[A_{1}=4(1+16K)A\text{, }A_{2}=4A^{2}(A_{1}+2)+16KA^{-1} \text{ and } C_{0}=10^{8}K^{3}A^{5}.\]
A key step in proving BHP is to estimate harmonic measure by Green functions, which is usually done by a so-called \emph{box argument}, see \cite[Theorem 2.4]{BB91} and \cite[Lemma 2]{Aik01}. 
\begin{lemma}\label{l-hmkm}
    There exists $\newconstant\label{c6} \in(0, \infty)$ such that for all $0<r<C_{0}^{-1}\diam(U,d), \xi \in \partial U$, there exist $\xi_r, \xi_r^{\prime} \in U$ that satisfy $d(\xi, \xi_r)=A_{1}r, \delta_U\left(\xi_r\right) \geq (2A)^{-1}A_{1}r, d\left(\xi_r, \xi_r^{\prime}\right)=A^{-1} r$ and
$$
\omega\left(x, U \cap \partial B_U(\xi, 2 r), B_U(\xi, 2 r)\right) \leq \oldconstant{c6} \frac{g_{B_U\left(\xi, A_2 r\right)}\left(x, \xi_r\right)}{g_{B_U\left(\xi, A_2 r\right)}\left(\xi_r^{\prime}, \xi_r\right)}, \quad\text{for $\mathcal{E}$-q.e. } x \in B_U(\xi, r) .
$$
\end{lemma}
\begin{proof}
Since $B_{U}(\xi,A_{1}r)\neq U$, we can find $\xi_{r}\in U$ by Lemma \ref{l-pick2}. Recall that in the proof of Lemma \ref{l-pick2}, $\xi_{r}$ is defined as the intersection of an $A$-uniform curve $\gamma$ and $\partial B(\xi,A_{1}r)$. We choose $\xi_{r}^{\prime}$ as $\gamma\cap \partial B(\xi_{r},A^{-1}r)$. Define \[g^{\prime}(z)=g_{B_U\left(\xi, A_2 r\right)}\left(z, \xi_r\right) \text { for } z \in B_U\left(\xi, A_2 r\right).\] Set $s=A^{-1}r$. Note that $B_U\left(\xi_r, s\right)=B\left(\xi_r, s\right) \subset U$. Since $B\left(\xi_r, s\right) \subset$ $B_U\left(\xi, A_2 r\right) \backslash B_U(\xi, 2 r)$, by applying the maximum principle in Theorem \ref{t-g-ex}, \[g^{\prime}(y) \leq \sup _{z \in \partial B\left(\xi_r, s\right)} g^{\prime}(z) \text { for all } y \in B_U(\xi, 2 r).\] As $\delta_U\left(\xi_r\right) \geq 2 A^{-1} (1+2K)r$, we have $B(\xi_{r},(1+2K)s)\subset B_{U}(\xi,A_{2}r)$ and therefore by Theorem \ref{t-g-cp} we have\begin{align}
    \max_{z \in \partial B\left(\xi_r, s\right)} g^{\prime}(z) & \leq \oldconstant{c1} \min_{z \in \partial B\left(\xi_r, s\right)} g^{\prime}(z)\label{e-2bp}\\ &\leq \oldconstant{c1} g^{\prime}\left(\xi_r^{\prime}\right).
    \end{align} Hence for $\varepsilon_{1}=\exp(-1)/(2\oldconstant{c1})>0$, there hold \[\varepsilon_1 \frac{g^{\prime}(y)}{g^{\prime}\left(\xi_r^{\prime}\right)}< \exp (-1), \quad \forall y \in B_U(\xi, 2 r).\]
With the aid of Lemma \ref{l-hc-log} and Theorem \ref{t-g-cp}, we can use the box argument as in \cite[Lemma 5.5]{BM19} and \cite[Lemma 5.3]{Lie15}.
For all non-negative integers $j$, define
\[    U_j:=\left\{x \in B_U\left(\xi, A_2 r\right): \exp\left(-2^{j+1}\right) \leq \varepsilon_1 \frac{g^{\prime}(x)}{g^{\prime}\left(\xi_r^{\prime}\right)}<\exp\left(-2^j\right)\right\},\, j\geq0,\]
so that $B_U(\xi, 2 r)=\bigcup_{j \geq 0} U_j \cap B_U(\xi, 2 r)$. Set $V_j=\bigcup_{k \geq j} U_k$. We claim that there exist $\newconstant\label{c5} \in(0, \infty)$ such that for all $j\geq0$ we have
\begin{equation}\label{e-hm-clm}
    w_{N_{{K}},\eta}\left(V_j \cap B_U(\xi, 2 r)\right) \leq 54(2\oldconstant{c1}^{2})^{-1/\oldconstant{c5}}\exp(\oldconstant{c5}^{-1})A\exp\left(-\frac{2^{j}}{\oldconstant{c5}}\right)r.
\end{equation} Let $x$ be an arbitrary point in $V_j \cap B_U(\xi, 2 r)$. By Lemma \ref{l-hc-log} there exist $\widetilde{\xi_{r}}\in \partial B(\xi_{r},A^{-1}r)$ and a $\delta_{\mathrm{H}}^{-1}$-Harnack chain of balls in $B_U\left(\xi, A_2 r\right) \backslash\left\{\xi_r\right\}$ from $\widetilde{\xi_{r}}$ to $x$, with length at most \[C(A,\delta_{\mathrm{H}}^{-1})+C(A,\delta_{\mathrm{H}}^{-1})\log\left(1+\frac{r}{\delta_{U}(x)}\right).\] Hence, by \eqref{e-hccp} and the fact that $2r/\delta_{U}(x)\geq1$ we conclude that  \begin{equation}\label{e-hmkm-c}
    \frac{g^{\prime}(x)}{g^{\prime}(\widetilde{\xi_{r}})}\geq 3^{-2\oldconstant{c5}}\left(\frac{\delta_{U}(x)}{r}\right)^{\oldconstant{c5}},
\end{equation}
where  $\oldconstant{c5}=C(A,\delta_{\mathrm{H}}^{-1})\log C_{\mathrm{H}}$.
By the definition of $V_{j}$,
\begin{align*}
    \exp\left(-2^j\right)>\frac{\exp(-1)}{2\oldconstant{c1}} \frac{g^{\prime}(x)}{g^{\prime}\left(\xi_r^{\prime}\right)} &\geq \frac{\exp(-1)}{2\oldconstant{c1}^{2}}  \frac{g^{\prime}(x)}{g^{\prime}(\widetilde{\xi_{r}})} \ \text{(by \eqref{e-2bp})}\\ &\geq \frac{\exp(-1)}{2\oldconstant{c1}^{2}3^{2\oldconstant{c5}}}\left(\frac{\delta_{U}(x)}{r}\right)^{\oldconstant{c5}} \ \text{(by \eqref{e-hmkm-c})},
\end{align*}
which implies \[\delta_{U}(x)\leq 9(2\oldconstant{c1}^{2})^{-1/\oldconstant{c5}}\exp(\oldconstant{c5}^{-1})\exp\left(-\frac{2^{j}}{\oldconstant{c5}}\right)r.\]
Therefore, we have for $j\geq 0$, \[V_{j}\cap B_{U}(\xi,2r)\subset\left\{x\in U: \delta_{U}(x)\leq9(2\oldconstant{c1}^{2})^{-1/\oldconstant{c5}}\exp(\oldconstant{c5}^{-1})\exp\left(-\frac{2^{j}}{\oldconstant{c5}}\right)r\right\}\]
which implies \eqref{e-hm-clm} by Lemma \ref{l-capw}.

Let $R_{0}=2r$ and for $j\geq1$, \[R_{j}=\left(2-\frac{6}{\pi^{2}}\sum_{k=1}^{j}\frac{1}{k^{2}}\right)r.\] Then $r<R_{j+1}<R_{j}<R_{0}=2r$ and $R_{j}\searrow r$. Let $\omega_{0}(\cdot)=\omega(\cdot,U\cap\partial B_{U}(\xi,2r),B_{U}(\xi,2r))$ and \[d_{j}=\begin{dcases}
    \sup_{x\in (U_{j}\cap B_{U}(\xi,R_{j}))\setminus\mathcal{N}}\frac{g^{\prime}(\xi_{r}^{\prime})\omega_{0}(x)}{g^{\prime}(x)},\quad & \text{if }U_{j}\cap B_{U}(\xi,R_{j})\neq\emptyset,\\0, \quad & \text{if }U_{j}\cap B_{U}(\xi,R_{j})=\emptyset.
\end{dcases}\] Here $\mathcal{N}$ is a $\mathcal{E}$-polar set which will be determined below. It suffices to show that $\sup_{j\geq0}d_{j}\leq\oldconstant{c6}<\infty$.  The proof is essentially the same as in \cite[Lemma 5.3]{Lie15}. For the reader’s convenience, we present the proof here. 

We proceed by iteration. Since $\omega_{0}\leq1$ on $B_{U}(\xi,2r)$. By definition of $U_{j}$ we have \[d_{0}\leq \exp(1)/(2\oldconstant{c1}) \text{ and } d_{1}\leq \exp(3)/(2\oldconstant{c1}).\]
Let $j\geq 2$ and let $x\in U_{j}\cap B_{U}(\xi,R_{j})$. We have \[\overline{V_{j}\cap B_{U}(\xi,R_{j-1})}\cap (U\cap \partial B_{U}(\xi,2r))=\emptyset,\] and \[\partial(V_{j}\cap B_{U}(\xi,R_{j-1}))\cap B_{U}(\xi,2r)=\left(U_{j-1}\cap\partial V_{j}\cap\overline{B_{U}(\xi,R_{j-1})}\right)\cup \left(V_{j}\cap \partial B_{U}(\xi,R_{j-1})\right)\]
where the two sets on right hand side are disjoint. By definition of $d_{j-1}$ and continuity of $\omega_{0}$ and $g^{\prime}$,
\[\omega_{0}(x)\leq d_{j-1}\frac{g^{\prime}(x)}{g^{\prime}(\xi_{r}^{\prime})}\ \text{ on }U_{j-1}\cap\partial V_{j}\cap\overline{B_{U}(\xi,R_{j-1})}.\] In view of the mean-value property of $\omega_{0}$ as in \eqref{e-hm-mp}, for $x\in (V_{j}\cap B_{U}(\xi,R_{j-1}))\setminus \mathcal{N}_{j}$ where $N_{j}$ is a $\mathcal{E}$-polar set, 
\begin{align}\label{e-hkmk-w}
    \omega_{0}(x)&=\left(\int_{U_{j-1}\cap\partial V_{j}\cap\overline{B_{U}(\xi,R_{j-1})}}+\int_{V_{j}\cap \partial B_{U}(\xi,R_{j-1})}\right)\omega_{0}(z) \omega(x, \dif z, V_{j}\cap B_{U}(\xi,R_{j-1}))\\&\leq d_{j-1}\frac{g^{\prime}(x)}{g^{\prime}(\xi_{r}^{\prime})}+\omega(x, V_{j}\cap \partial B_{U}(\xi,R_{j-1}), V_{j}\cap B_{U}(\xi,R_{j-1})).
\end{align}
To estimate the second term in \eqref{e-hkmk-w}, we notice that \[B_{U}(\xi,2r)\cap B(x,R_{j-1}-R_{j})\subset B_{U}(\xi,R_{j-1})\] and apply domain monotonicity of harmonic measure \eqref{e-hm-dm},
\begin{align}\label{e-hkmk-x}
    &\omega(x, V_{j}\cap \partial B_{U}(\xi,R_{j-1}), V_{j}\cap B_{U}(\xi,R_{j-1}))\\ \leq\ & \omega(x, V_{j}\cap \partial\left(B_{U}(\xi,2r)\cap B(x,R_{j-1}-R_{j})\right), V_{j}\cap B_{U}(\xi,2r)\cap B(x,R_{j-1}-R_{j}))\\ \leq\ & \omega(x, V_{j}\cap B_{U}(\xi,2r)\cap \partial B(x,R_{j-1}-R_{j}), V_{j}\cap B_{U}(\xi,2r)\cap B(x,R_{j-1}-R_{j})) \\ \leq\ & \exp\left(2-\frac{a_{1}\left(R_{j-1}-R_{j}\right)}{w_{N_{K},\eta}(V_{j}\cap B_{U}(\xi,2r))}\right)  \ \text{(by Lemma \ref{l-hm-cw})} \\ \leq\ & \exp\left(2-\frac{a_{1}j^{-2}\exp\left({2^{j}}/{\oldconstant{c5}}\right)}{9\pi^{2}(2\oldconstant{c1}^{2})^{-1/\oldconstant{c5}}\exp(\oldconstant{c5}^{-1})A}\right)  \ \text{(by \eqref{e-hm-clm})},
\end{align}
where the second inequality uses the fact that \[\partial\left(B_{U}(\xi,2r)\cap B(x,R_{j-1}-R_{j})\right)\subset B_{U}(\xi,2r)\cap \partial B(x,R_{j-1}-R_{j}).\] Moreover, since $x\in U_{j}\cap B_{U}(\xi,R_{j})$, we have by definition of $U_{j}$ that \[1\leq (2\exp(1)\oldconstant{c1})^{-1}\frac{g^{\prime}(x)}{g^{\prime}(\xi_{r}^{\prime})}\exp(2^{j+1}).\]Combining \eqref{e-hkmk-w} and \eqref{e-hkmk-x} we have \begin{equation*}
    \omega_{0}(x)\leq \left(d_{j-1}+(2\exp(1)\oldconstant{c1})^{-1}\exp\left(2+2^{j+1}-\frac{a_{1}j^{-2}\exp\left({2^{j}}/{\oldconstant{c5}}\right)}{9\pi^{2}(2\oldconstant{c1}^{2})^{-1/\oldconstant{c5}}\exp(\oldconstant{c5}^{-1})A}\right)\right)\frac{g^{\prime}(x)}{g^{\prime}(\xi_{r}^{\prime})}.
\end{equation*}
Dividing both sides by $\frac{g^{\prime}(x)}{g^{\prime}(\xi_{r}^{\prime})}$ and taking supremum over all $x\in (U_{j}\cap B_{U}(\xi,R_{j}))\setminus \mathcal{N}$, where $\mathcal{N}:=\bigcup_{j}\mathcal{N}_{j}$, we conclude that \[d_{j}-d_{j-1}\leq\frac{\exp(1)}{2\oldconstant{c1}}\exp\left(2^{j+1}-\frac{a_{1}j^{-2}\exp\left({2^{j}}/{\oldconstant{c5}}\right)}{9\pi^{2}(2\oldconstant{c1}^{2})^{-1/\oldconstant{c5}}\exp(\oldconstant{c5}^{-1})A}\right).\]
Thus for all $j\geq2$, \begin{align*}
    d_{j}&\leq d_{1}+\frac{\exp(1)}{2\oldconstant{c1}}\sum_{j=2}^{\infty}\exp\left(2^{j+1}-\frac{a_{1}j^{-2}\exp\left({2^{j}}/{\oldconstant{c5}}\right)}{9\pi^{2}(2\oldconstant{c1}^{2})^{-1/\oldconstant{c5}}\exp(\oldconstant{c5}^{-1})A}\right)\\ &\leq \frac{\exp(3)}{2\oldconstant{c1}}+\frac{\exp(1)}{2\oldconstant{c1}}\sum_{j=2}^{\infty}\exp\left(2^{j+1}-\frac{a_{1}j^{-2}\exp\left({2^{j}}/{\oldconstant{c5}}\right)}{9\pi^{2}(2\oldconstant{c1}^{2})^{-1/\oldconstant{c5}}\exp(\oldconstant{c5}^{-1})A}\right):=\oldconstant{c6}<\infty.
\end{align*}
and we complete the proof.
\end{proof}
We define 
\begin{equation*}
  A_{3}=\max(A_{2},2+12A^{2},8) .
\end{equation*}

\begin{notation}\label{n-4p}
    For each fixed $\xi\in\partial U$, $0<r< (2A_{3})^{-1}\diam(U,d)$, we denote $x_{\xi}^{\star}\in U\cap \partial B_{U}(\xi,r)$ and $y_{\xi}^{\star}\in U\cap \partial B_{U}(\xi,A_{3}r)$ given by Lemma \ref{l-pick2} such that $\delta_{U}(x_{\xi}^{\star})\geq (2A)^{-1}r$ and $\delta_{U}(y_{\xi}^{\star})\geq (2A)^{-1}A_{3}r$. Denote $\gamma_{\xi}$ to be an $A$-uniform curve from $y_{\xi}^{\star}$ to $x_{\xi}^{\star}$, and $z_{\xi}^{\star}$ to be the last point of this curve which is on $\partial B(y_{\xi}^{\star},A^{-1}r)$. We set \[F(\xi)=B_{U}(\xi,(A_{3}+3)r)\setminus B_{U}(\xi,(A_{3}-3)r).\]  
\end{notation}
Notice that by Lemma \ref{l-dist},\begin{equation}
\delta_{U}(z)\geq(2A)^{-2}r\ \text{for all } z\in\gamma_{\xi}\text{ and } \delta_{U}(y_{\xi}^{\star})\geq (2A)^{-1}A_{3}r>2r, 
\end{equation}
so that $B(y_{\xi}^{\star},2r)\subset U$, and \begin{equation}\label{e-diam}
    \diam(\gamma_{\xi})\leq Ad(x_{\xi}^{\star},y_{\xi}^{\star})\leq 2A_{3}Ar.
\end{equation}
\noindent
\begin{figure}
\trimbox{0cm 4.5cm 0cm 0cm}{
\begin{tikzpicture}
\coordinate [label=below:$\xi$] (origin) at (0,0);
\node at (6.5,3)  {$U$};
\node at (7.8,0)  {$\partial U$};
\draw[name path=boundary, thick] plot [smooth, tension=1] coordinates{(-7.4,0) (-7.2,0.15) (-6.8, 0)  (-5.9, 0.05)  (-5, 0) (-4,0.1) (-3.5, 0.2) (-2.6, 0) (-2,0.1) (-1.6, 0.15)(-1.2,0.13) (-0.8,0.07) (0,0) (0.7,-0.05) (1.1, 0.01) (1.8, 0.08) (2.2, 0.04) (3, 0) (4.2,0.1) (5.5, -0.1)  (6.3, 0.13) (6.9, 0) (7.2, -0.07)};
\filldraw[color=black, fill=black, very thick] (origin) circle (0.5pt);
\def\radiusa{0.6}
\path[name path=circle0.6] (origin) circle (\radiusa); 
\path [name intersections={of=boundary and circle0.6}] coordinate (l1) at (intersection-1) coordinate [label=below:$r$] (r1) at (intersection-2);
 \begin{scope}
        \pic [draw, angle radius=\radiusa cm] {angle=r1--origin--l1};
    \end{scope}
 \coordinate [label=above:$\scriptstyle{x^{\star}_{\xi}}$] (xxi) at (50:\radiusa);
 \filldraw[color=black, fill=black, very thick] (xxi) circle (0.3pt);
\def\radiusb{1.2}
\path[name path=circle1.2] (origin) circle (\radiusb); 
\path [name intersections={of=boundary and circle1.2}] coordinate (l2) at (intersection-1) coordinate [label=below:$2r$] (r2) at (intersection-2);
 \begin{scope}
        \pic [draw, angle radius=\radiusb cm] {angle=r2--origin--l2};
    \end{scope}
\def\radiusc{2.4}
\path[name path=circle2.4] (origin) circle (\radiusc); 
\path [name intersections={of=boundary and circle2.4}] coordinate (l3) at (intersection-1) coordinate [label=below:$A_{1}r$] (r3) at (intersection-2);
 \begin{scope}
        \pic [draw, angle radius=\radiusc cm] {angle=r3--origin--l3};
    \end{scope}
 \coordinate [label=above:$\scriptstyle{\xi_{r}}$] (xir) at (70:\radiusc);
 \filldraw[color=black, fill=black, very thick] (xir) circle (0.3pt);
\def\radiusd{3.8}
\path[name path=circle4] (origin) circle (\radiusd); 
\path [name intersections={of=boundary and circle4}] coordinate (l4) at (intersection-1) coordinate [label=below:$A_{3}r$] (r4) at (intersection-2);
 \begin{scope}
        \pic [draw, angle radius=\radiusd cm] {angle=r4--origin--l4};
    \end{scope}
 \coordinate [label=above:$\scriptstyle{y^{\star}_{\xi}}$] (yxi) at (120:\radiusd);
 \filldraw[color=black, fill=black, very thick] (yxi) circle (0.3pt);
\def\radiuse{5}
\path[name path=circle6] (origin) circle (\radiuse); 
\path [name intersections={of=boundary and circle6}] coordinate (l5) at (intersection-1) coordinate [label=below:$A_{4}r$] (r5) at (intersection-2);
 \begin{scope}
        \pic [draw, angle radius=\radiuse cm] {angle=r5--origin--l5};
    \end{scope}
 \draw [dashed, radius=0.45, name path=circleyxi] (yxi) circle;
 \draw[name path=uni-curv] plot [smooth, tension=1] coordinates{(yxi) (-0.5, 2) (xxi)};
 \path [name intersections={of=uni-curv and circleyxi}] coordinate [label=below:$\scriptstyle{z^{\star}_{\xi}}$] (zxi) at (intersection-1)  ; 
 \filldraw[color=black, fill=black, very thick] (zxi) circle (0.3pt);
 
\draw [dashed, radius=0.45, name path=circlexir] (xir) circle;
\tikzset{shift={(xir)}};
\coordinate [label=above: $\scriptstyle{\xi_{r}^{\prime}}$] (xirp) at (60: 0.45);
\filldraw[color=black, fill=black, very thick] (xirp) circle (0.3pt);
\end{tikzpicture}}
\caption{Five specified points $x_{\xi}^{\star},y_{\xi}^{\star},z_{\xi}^{\star},\xi_{r}, \xi_{r}^{\prime}$.}
\end{figure}
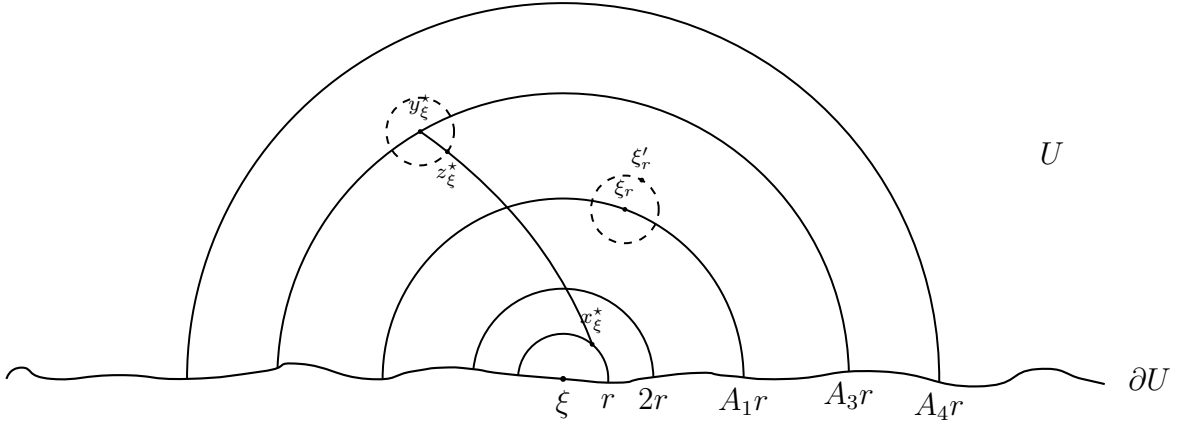

Define
\[A_{4}=A_{3}+2AA_{3}+(1+16K)A^{-1}.\] {The following lemma, which compares Green functions at some specified points is useful to give estimates on a region. Note that  \cite[Lemma 5.10 and Lemma 5.11]{BM19} also use such estimates without proof, since in the context of \cite{BM19}, such estimates naturally hold as the space is assumed to be geodesic, allowing any two points to be connected by a rectifiable curve, hence are connected through a Harnack chain with small radius. However, in the current framework, these estimates are not apparent.}
\begin{lemma}\label{l-g-4po}
    There exists  $\newconstant \label{c10}$ such that for any $\xi\in\partial U$, $0<r< (2A_{3})^{-1}\diam(U,d)$, $D:=B_{U}(\xi,A_{4}r)$, and any two points $\eta, \eta^{\prime}\in B_{U}(\xi,A_{2}r)$ with $\eta\in U\cap \partial B(\xi,A_{1}r)$, $\delta_{U}(\eta)\geq 2(1+16K)A^{-1}r$ and $\eta^{\prime}\in\partial B(\eta,A^{-1}r)$, we have
   \begin{equation}\label{e-4po-1}
        \oldconstant{c10}^{-1}g_{B_U\left(\xi, A_2 r\right)}\left(\eta^{\prime}, \eta\right)\leq g_{D}(x_{\xi}^{\star},y_{\xi}^{\star})\leq \oldconstant{c10}g_{B_U\left(\xi, A_2 r\right)}\left(\eta^{\prime}, \eta\right)
    \end{equation}
    and 
    \begin{equation}\label{e-4po-2}
        g_{B_U\left(\xi, A_2 r\right)}\left(x, \eta\right)\leq \oldconstant{c10}g_{D}(x,y_{\xi}^{\star})\quad\text{for all }x\in B_{U}(\xi,2r).
    \end{equation}
\end{lemma}
\begin{proof}
\begin{enumerate}[label=\textup{(\arabic*)},align=left,leftmargin=*,topsep=5pt,parsep=0pt,itemsep=2pt]
        \item For the proof of \eqref{e-4po-1}, let $\gamma_{\xi}^{\prime}$ be an $A$-uniform curve from $y_{\xi}^{\star}$ to $\eta$. By definition of uniform curve, \begin{equation}\label{e-diamp}
            \diam(\gamma_{\xi}^{\prime})\leq Ad(y_{\xi}^{\star},\eta)\leq 2A_{3}Ar,
        \end{equation} and 
        \begin{equation}\label{e-distp}
            \delta_{U}(z)\geq(2A)^{-1}\min(2A^{-1}A_{3}r,2(1+16K)A^{-1}r)\geq A^{-2}r\quad \text{for all } z\in \gamma_{\xi}^{\prime}.
        \end{equation}  Let $w_{\xi}^{\star}$ denote the last point of $\gamma_{\xi}^{\prime}$ which is on $\partial B(y_{\xi}^{\star},A^{-1}r)$ and $\eta^{\prime \prime}$ denote the first point of $\gamma_{\xi}^{\prime}$ which is on $\partial B(\eta,A^{-1}r)$. Choose a $\delta_{\mathrm{H}}^{-1}$-Harnack chain connecting $x_{\xi}^{\star}$ and $z_{\xi}^{\star}$ centred in $\gamma_{\xi}$ with radius $(3A)^{-2}r$. By \eqref{e-diam} and Proposition \ref{p-hc-p}, this Harnack chain is contained in $D \setminus\{y_{\xi}^{\star}\}$ and has length less than $\newconstant\label{c12}:=C_{D}(1+6A_{3}A^{3}\delta_{\mathrm{H}}^{-1})^{\alpha}$. Since $g_{D}(\cdot,y_{\xi}^{\star})$ is $\mathcal{E}$-harmonic in $D\setminus\{y_{\xi}^{\star}\}$, we have by \eqref{e-hccp} that for $\newconstant\label{c13}=\exp(\oldconstant{c12}\log  C_{\mathrm{H}})$,
\begin{equation}\label{e-4po-3}
        \oldconstant{c13}^{-1} g_{D}(x_{\xi}^{\star},y_{\xi}^{\star})\leq g_{D}(z_{\xi}^{\star},y_{\xi}^{\star}) \leq \oldconstant{c13} g_{D}(x_{\xi}^{\star},y_{\xi}^{\star}).
    \end{equation}
    Similarly, we have for some $\newconstant\label{c14}$ that,
     \begin{equation}\label{e-4po-4}
        \oldconstant{c14}^{-1} g_{D}(\eta,\eta^{\prime \prime})\leq g_{D}(\eta,y_{\xi}^{\star}) \leq \oldconstant{c14} g_{D}(\eta,\eta^{\prime \prime}),
    \end{equation} 
    \begin{equation}\label{e-4po-5}
        \oldconstant{c14}^{-1} g_{D}(w_{\xi}^{\star},y_{\xi}^{\star})\leq g_{D}(\eta,y_{\xi}^{\star}) \leq \oldconstant{c14} g_{D}(w_{\xi}^{\star},y_{\xi}^{\star}).
    \end{equation}
As $z_{\xi}^{\star}$ and $w_{\xi}^{\star}$ are both in $\partial B(y_{\xi}^{\star},A^{-1}r)$, and $B(y_{\xi}^{\star},(1+2K)A^{-1}r)\subset B_{U}(\xi,A_{4}r)=D$, we conclude by Theorem \ref{t-g-cp} that \begin{equation} \label{e-4po-6}
       \oldconstant{c1}^{-1} g_{D}(z_{\xi}^{\star},y_{\xi}^{\star})\leq g_{D}(w_{\xi}^{\star},y_{\xi}^{\star})\leq \oldconstant{c1} g_{D}(z_{\xi}^{\star},y_{\xi}^{\star})
    \end{equation}
    Similarly, as $\eta^{\prime}$ and $\eta^{\prime \prime}$ are both in $\partial B(\eta,A^{-1}r)$, and $B(\eta,(1+2K)A^{-1}r)\subset B_{U}(\xi,A_{4}r)=D$, we have \begin{equation}\label{e-4po-7}
        \oldconstant{c1}^{-1} g_{D}(\eta,\eta^{\prime \prime})\leq g_{D}(\eta,\eta^{\prime}) \leq \oldconstant{c1} g_{D}(\eta,\eta^{\prime \prime})
    \end{equation}
    Combining \eqref{e-4po-3}, \eqref{e-4po-4}, \eqref{e-4po-5}, \eqref{e-4po-6} and \eqref{e-4po-7}, we have for $\newconstant\label{c16}=\oldconstant{c13}\oldconstant{c14}^{2}\oldconstant{c1}^{2}$, \begin{equation}\label{e-4po-8}
    \oldconstant{c16}^{-1}g_{D}\left(\eta,\eta^{\prime}\right)\leq g_{D}(x_{\xi}^{\star},y_{\xi}^{\star})\leq \oldconstant{c16}g_{D}\left(\eta,\eta^{\prime}\right).
    \end{equation}
    By the domain monotonicity of Green function and \eqref{e-4po-8}, we have \begin{equation}\label{e-4po-9}
        g_{B_U\left(\xi, A_2 r\right)}\left(\eta,\eta^{\prime}\right)\leq g_{D}\left(\eta,\eta^{\prime}\right)\leq \oldconstant{c16}g_{D}(x_{\xi}^{\star},y_{\xi}^{\star}).
    \end{equation}
Since $B(\eta,16KA^{-1}r)\subset B_U\left(\xi, A_2 r\right)$, by Theorem \ref{t-g-cp} and domain monotonicity of Green function again we have \begin{align}\label{e-4po-10}
       g_{D}(x_{\xi}^{\star},y_{\xi}^{\star})&\leq \oldconstant{c16} g_{D}\left(\eta,\eta^{\prime}\right) \\ &\leq \oldconstant{c16} g_{B(\eta,(A_{2}+A_{4})r)}\left(\eta,\eta^{\prime}\right) \\ &\leq \oldconstant{c16} \oldconstant{c4} g_{B(\eta,16KA^{-1}r)}\left(\eta,\eta^{\prime}\right) \\ & \leq \oldconstant{c16}\oldconstant{c4} g_{B_U\left(\xi, A_2 r\right)}\left(\eta,\eta^{\prime}\right).
       \end{align}
       Combining \eqref{e-4po-9} and \eqref{e-4po-10}, we have \eqref{e-4po-1}.
       \item Since $g_{B_U\left(\xi, A_2 r\right)}\left(x, \eta \right)\leq g_{D}(x,\eta)$ by domain monotonicity, we only need to prove that $g_{D}(x,\eta)\leq \oldconstant{c10} g_{D}(x,y_{\xi}^{\star})$.
       \begin{enumerate}[label=\textup{(\roman*)},align=left,leftmargin=*,topsep=5pt,parsep=0pt,itemsep=2pt]
           \item If $x\in B_{U}(\xi,2r)$ and $\delta_{U}(x)<(2A^{2})^{-1}r$. We can pick a $\delta_{\mathrm{H}}^{-1}$-Harnack chain from $y_{\xi}^{\star}$ to $\eta$ centred in $\gamma_{\xi}^{\prime}$ with radius $(2A^{2})^{-1}r$; this chain is contained in $D\setminus\{x\}$ since $d(x,z)\geq \delta_{U}(z)-\delta_{U}(x)>A^{-2}r-(2A^{2})^{-1}r=(2A^{2})^{-1}r$ for all $z\in\gamma_{\xi}^{\prime}$, and have length less than $\newconstant\label{c18}:=C_{D}(1+4\delta_{\mathrm{H}}^{-1}A_{3}A^{3})^{\alpha}$, by Proposition \ref{p-hc-p}. As $g_{D}(x,\cdot)$ is $\mathcal{E}$-harmonic in $D\setminus\{x\}$ we conclude that $g_{D}(x,\eta)\leq \exp(\oldconstant{c18}\log C_{\mathrm{H}}) g_{D}(x,y_{\xi}^{\star})$. 
           \item If $x\in B_{U}(\xi,2r)$ and $\delta_{U}(x)\geq(2A^{2})^{-1}r$. By joining an $A$-uniform curve from $\eta$ to $x$ and a similar argument as in the proof of \eqref{e-4po-3}, we conclude that for some constant $\newconstant\label{c19}$ and $\newconstant\label{c20}$, \[ g_{D}(x,\eta)\leq \oldconstant{c19} g_{D}(\eta,\eta^{\prime}) \text{ and } g_{D}(x,y_{\xi}^{\star})\geq \oldconstant{c20} g_{D}(z_{\xi}^{\star},y_{\xi}^{\star})\]
           From \eqref{e-4po-8} we have $g_{D}(x,\eta)\leq \oldconstant{c16} \oldconstant{c19} \oldconstant{c20}^{-1}g_{D}(x,y_{\xi}^{\star})$.
       \end{enumerate}
 Combining the above discussion, \eqref{e-4po-1} and \eqref{e-4po-2} hold with \[\oldconstant{c10}=\max\left( \oldconstant{c16} \oldconstant{c19} \oldconstant{c20}^{-1},\exp(\oldconstant{c18}\log C_{\mathrm{H}}), \oldconstant{c16}\oldconstant{c4} \right).\]
    \end{enumerate}
\end{proof}

{The rest of the proof is the same as \cite[Lemma 5.8-5.11]{BM19} with suitable modifications of the chain argument.} The basic idea is to first obtain BHP for Green functions, then represent general harmonic functions by the formula in Proposition \ref{p-blyg}. 

\begin{lemma}[See{\cite[Lemma 5.8]{BM19}}]\label{l-g-cp-bd}
    Let $\xi\in\partial U$, $0<r< C_{0}^{-1}\diam(U,d)$ and $D=B_{U}(\xi,A_{4}r)$. There exists $\newconstant\label{c7}$ such that 
    \begin{equation}\label{e-g-cp-bd}
    \oldconstant{c7}^{-1} g_{D}(x,y)\leq\frac{g_{D}(x_{\xi}^{\star},y)}{g_{D}(x_{\xi}^{\star},y_{\xi}^{\star})} g_{D}(x,y_{\xi}^{\star})\leq \oldconstant{c7}g_{D}(x,y)    
    \end{equation}
    for all $x\in B_{U}(\xi,r)$ and $y\in U\cap\partial B_{U}(\xi,A_{3}r)$ with $\delta_{U}(y)\geq (24A^{3})^{-1}r$.
\end{lemma}
\begin{proof}
    Fix $x\in B_{U}(\xi,r)$. Define\[u_{1}(y)=g_{D}(x,y),\quad v_{1}(y)=\frac{g_{D}(x_{\xi}^{\star},y)}{g_{D}(x_{\xi}^{\star},y_{\xi}^{\star})} g_{D}(x,y_{\xi}^{\star}),\ y\in D\setminus\{x,x_{\xi}^{\star}\}.\]The functions $u_{1}$ and $v_{1}$ are $\mathcal{E}$-harmonic in $D\setminus\{x,x_{\xi}^{\star}\}$ and $u_{1}(y_{\xi}^{\star})=v_{1}(y_{\xi}^{\star})$. Let $\gamma$ be an $A$-uniform curve from $y$ to $y_{\xi}^{\star}$ with $\diam(\gamma)\leq 2A_{3}Ar$. By Lemma \ref{l-curv-b}, $\gamma\subset U\setminus B_{U}(\xi, A_{3}(A+1)^{-1}r)\subset U\setminus B_{U}(\xi, 2r)$. Also by Lemma \ref{l-dist}, $\delta_{U}(z)\geq (2A)^{-1}\min(\delta_{U}(y),\delta_{U}(y_{\xi}^{\star}))\geq(48A^{4})^{-1}r$. By Proposition \ref{p-hc-p} we can find a $\delta_{\mathrm{H}}^{-1}$-Harnack chain of balls in $D\setminus\{x,x_{\xi}^{\star}\}$ of radius $(49A^{4})^{-1}r$, connecting $y$ and $y_{\xi}^{\star}$ and length less than $\newconstant\label{c35}=C_{D}(1+98A^{5}A_{3}\delta_{\mathrm{H}}^{-1})^{\alpha}$. Therefore, \eqref{e-g-cp-bd} holds with $\oldconstant{c7}=\exp(2\oldconstant{c35}\log C_{\mathrm{H}})$ by \eqref{e-hccp}. 
\end{proof}
{We then estimate Green function when $y$ is near to the boundary.}
\begin{lemma}[See{\cite[Lemma 5.9]{BM19}}]\label{l-g-cp-l}
    Let $\xi\in\partial U$, $0<r< C_{0}^{-1}\diam(U,d)$ and $D=B_{U}(\xi,A_{4}r)$. There exists $\newconstant\label{c8}$ such that 
    \begin{equation}\label{e-g-cp-l}
    g_{D}(x,y)\geq \oldconstant{c8}^{-1}\frac{g_{D}(x_{\xi}^{\star},y)}{g_{D}(x_{\xi}^{\star},y_{\xi}^{\star})} g_{D}(x,y_{\xi}^{\star})
    \end{equation}
    for all $x\in B_{U}(\xi,r)$ and $y\in U\cap\partial B_{U}(\xi,A_{3}r)$ with $\delta_{U}(y)<(24A^{3})^{-1}r$.
\end{lemma}
\begin{proof}
    Fix $y\in U\cap\partial B_{U}(\xi,A_{3}r)$ with $\delta_{U}(y)<(24A^{3})^{-1}r$. Define
    \[u_{2}(x)=g_{D}(x,y),\quad v_{2}(x)=\frac{g_{D}(x_{\xi}^{\star},y)}{g_{D}(x_{\xi}^{\star},y_{\xi}^{\star})} g_{D}(x,y_{\xi}^{\star}),\ x\in D\setminus\{y,y_{\xi}^{\star}\}.\]
    According to Theorem \ref{t-g-ex}, $u_{2}$ and $v_{2}$ are $\mathcal{E}$-harmonic in $D\setminus\{y,y_{\xi}^{\star}\}$ and $u_{2}(x_{\xi}^{\star})=v_{2}(x_{\xi}^{\star}).$ Let $\gamma_{\xi}$ and $z_{\xi}^{\star}$ be as in Notation \ref{n-4p}. By Lemma \ref{l-dist}, $\delta_{U}(z)\geq(2A)^{-1}\min(\delta_{U}(x_{\xi}^{\star}),\delta_{U}(y_{\xi}^{\star}))\geq(12A^{3})^{-1}r$ for all $z\in\gamma_{\xi}$. Since $\delta_{U}(y)<(24A^{3})^{-1}r$, \[d(z,y)\geq \delta_{U}(z)- \delta_{U}(y)> (24A^{3})^{-1}r,\quad\forall z\in\gamma_{\xi}.\] By Proposition \ref{p-hc-p}, there exists a $\delta_{\mathrm{H}}^{-1}$-Harnack chain of balls centred in $\gamma_{\xi}$ with radius $(25A^{3})^{-1}r$, connecting $x_{\xi}^{\star}$ and $z_{\xi}^{\star}$ and contained in $U\setminus\{y,y_{\xi}^{\star}\}$ whose length of this chain is less than $\newconstant\label{c30}:=C_{D}(1+50A^{4}A_{3}\delta_{\mathrm{H}}^{-1})^{\alpha}$. Using \eqref{e-hccp} we deduce that \[\exp(-\oldconstant{c30}\log C_{\mathrm{H}})v_{2}(z_{\xi}^{\star})\leq v_{2}(x_{\xi}^{\star})=u_{2}(x_{\xi}^{\star})\leq \exp(\oldconstant{c30}\log C_{\mathrm{H}})u_{2}(z_{\xi}^{\star}).\]
 For any $z\in \partial B(y_{\xi}^{\star},A^{-1}r)$, by triangle inequality, $\delta_{U}(z)\geq\delta_{U}(y_{\xi}^{\star})- A^{-1}r\geq ((2A)^{-1}A_{3}-A^{-1})r$. Thus there exists a $\delta_{\mathrm{H}}^{-1}$-Harnack chain of balls in $D\setminus B(y, (24A^{3})^{-1}r)$ connecting $z$ and $z_{\xi}^{\star}$ with finite length independent of $z$.  Therefore for some $\newconstant\label{c40}$ we have $u_{2}(z_{\xi}^{\star})\leq \oldconstant{c40}u_{2}(z)$. By Theorem \ref{t-g-cp}, for all $z\in \partial B(y_{\xi}^{\star},A^{-1}r)$ we have, $\oldconstant{c1}^{-1}v_{2}(z)\leq v_{2}(z_{\xi}^{\star})$. Consequently, for all $z\in \partial B(y_{\xi}^{\star},A^{-1}r)$,  \begin{align*}
 g_{D}(z,y)=u_{2}(z)&\geq \oldconstant{c1}^{-1}\oldconstant{c40}^{-1}\exp(-2\oldconstant{c30}\log C_{\mathrm{H}})v_{2}(z)\\ &=\left(\oldconstant{c1}^{-1}\oldconstant{c40}^{-1}\exp(-2\oldconstant{c30}\log C_{\mathrm{H}})\frac{g_{D}(x_{\xi}^{\star},y)}{g_{D}(x_{\xi}^{\star},y_{\xi}^{\star})}\right) g_{D}(z,y_{\xi}^{\star}).
 \end{align*}
 Since $B(y^{\star}_{\xi},(1+2K)A^{-1}r)\subset D$, we can apply Lemma \ref{l-g-mp2} with \[c_{0}:=\left(\oldconstant{c1}^{-1}\oldconstant{c40}^{-1}\exp(-2\oldconstant{c30}\log C_{\mathrm{H}})\frac{g_{D}(x_{\xi}^{\star},y)}{g_{D}(x_{\xi}^{\star},y_{\xi}^{\star})}\right)\] and deduce that \eqref{e-g-cp-l} holds for all $x\in D\setminus(\{y\}\cup B(y_{\xi}^{\star},A^{-1}r))\supset B_{U}(\xi,r)$ with $\oldconstant{c8}=\oldconstant{c1}^{2}\exp(2\oldconstant{c30}\log C_{\mathrm{H}})$.
\end{proof}

Define \[A_{5}=A_{3}+A_{4}.\]

The following lemma is a version of \emph{Carleson estimate} for Green function.
\begin{lemma}[See{\cite[Lemma 5.10]{BM19}}]\label{l-g-cp-anu} 
Let $\xi\in\partial U$, $0<r< C_{0}^{-1}\diam(U,d)$ and $D=B_{U}(\xi,A_{4}r)$. Then there exists $\newconstant\label{c31}$ such that \begin{equation} 
    g_{D}(x,z)\leq \oldconstant{c31}g_{D}(x,y_{\xi}^{\star})\quad\text{for all }x\in B_{U}(\xi,2r),z\in F(\xi).
\end{equation}
\end{lemma}
\begin{proof}
    We first prove that \begin{equation}\label{a-g-cp-au0}
        g_{D}(x,y)\leq \oldconstant{c31} g_{D}(x_{\xi}^{\star},y_{\xi}^{\star})\quad\text{for all }x\in B_{U}(\xi,2r),y\in F(\xi).
    \end{equation}
    Let $x\in B_{U}(\xi,2r),y\in F(\xi)$. By triangle inequality, \[d(x,y)\geq d(y,\xi)-d(x,\xi)\geq(A_{3}-5)r.\]  Since $D\subset B(x_{\xi}^{\star},(1+A_{4})r)\subset B(x_{\xi}^{\star},(3+2K)(1+A_{4})r)$, and both $d(x,y)$ and $d(x_{\xi}^{\star},y_{\xi}^{\star})$ are not less than $(A_{3}-5)r$, we have for all $x\in B_{U}(\xi,2r),y\in F(\xi)$, \begin{align}
        g_{D}(x,y)&\leq g_{B(x_{\xi}^{\star},(3+2K)(1+A_{4})r)}(x,y) \ \text{ (by domain monotonicity of Green function)}\\ &\leq \oldconstant{c2}g_{B(x_{\xi}^{\star},(3+2K)(1+A_{4})r)}(x_{\xi}^{\star},y_{\xi}^{\star})  \ \text{ (by Theorem \ref{t-g-cp})}\\ &\leq \oldconstant{c4}\oldconstant{c2}g_{D}(x_{\xi}^{\star},y_{\xi}^{\star})  \ \text{ (by Theorem \ref{t-g-cp} and domain monotonicity)}. \end{align} This yields \eqref{a-g-cp-au0}. By the continuity of Green function, \eqref{a-g-cp-au0} can be extended as follows:\begin{equation}\label{a-g-cp-au1}
         g_{D}(x,y)\leq \oldconstant{c4}\oldconstant{c2} g_{D}(x_{\xi}^{\star},y_{\xi}^{\star})\quad\text{for all }x\in U\cap\overline{B_{U}(\xi,2r)},y\in F(\xi).
    \end{equation}
    Let $z\in F(\xi)$. By Lemma \ref{l-g-hmk}, for $\mathcal{E}$-q.e. $x\in B_{U}(\xi,2r)$\begin{align}
        g_{D}(x,z)&\leq \omega(x,U\cap\partial B_{U}(\xi,2r), B_{U}(\xi,2r))\sup_{y\in U\cap \partial B_{U}(\xi,2r)}g_{D}(y,z)\ \text{ (by Lemma \ref{l-g-hmk})} \\ &\leq  \oldconstant{c4}\oldconstant{c2} \oldconstant{c6}  \frac{g_{B_U\left(\xi, A_2 r\right)}\left(x, \xi_r\right)}{g_{B_U\left(\xi, A_2 r\right)}\left(\xi_r^{\prime}, \xi_r\right)}g_{D}(x_{\xi}^{\star},y_{\xi}^{\star})\ \text{(by Lemma \ref{l-hmkm} and \eqref{a-g-cp-au1})}.
        \\ &\leq  \oldconstant{c4}\oldconstant{c2} \oldconstant{c6} \oldconstant{c10}^{2} g_{D}(x,y_{\xi}^{\star})\ \text{(by Lemma \ref{l-g-4po} with $\eta=\xi_{r}$ and $\eta^{\prime}=\xi_{r}^{\prime}$)}.\label{e-g-au2}
    \end{align}
By the continuity of $g_{D}(\cdot, z)$ and $g_{D}(\cdot, y_{\xi}^{\star})$, we can extend \eqref{e-g-au2} to all $x\in B_{U}(\xi, 2r)$. This gives the lemma with $\oldconstant{c31}=\oldconstant{c4}\oldconstant{c2} \oldconstant{c6} \oldconstant{c10}^{2}$.
\end{proof}
{The lemma presented below complements Lemma \ref{l-g-cp-l}, offering an alternative estimate in the opposite direction, and gives estimates of Green function when $y$ is near to the boundary.}
\begin{lemma}[See{\cite[Lemma 5.11]{BM19}}]\label{l-g-cp-u}
    Let $\xi\in\partial U$, $0<r< C_{0}^{-1}\diam(U,d)$ and $D=B_{U}(\xi,A_{4}r)$. There exists $\newconstant\label{c32}$ such that 
    \begin{equation}\label{e-g-cp-u}
    g_{D}(x,y)\leq \oldconstant{c32}\frac{g_{D}(x_{\xi}^{\star},y)}{g_{D}(x_{\xi}^{\star},y_{\xi}^{\star})} g_{D}(x,y_{\xi}^{\star})
    \end{equation}
    for all $x\in B_{U}(\xi,r)$ and $y\in U\cap\partial B_{U}(\xi,A_{3}r)$ with $\delta_{U}(y)<(24A^{3})^{-1}r$.
\end{lemma}
\begin{proof}
   Fix $y$ and let $\zeta\in\partial U$ be a point such that $d(y,\zeta)<(24A^{3})^{-1}r$, and let $\zeta_{r}$ and $\zeta_{r}^{\prime}$ be the points given by Lemma \ref{l-hmkm}  corresponding to the boundary point $\zeta$. Since $B_{U}(\zeta,2r)\subset F(\xi)$ and  $g_{D}(x,\cdot)$ is $\mathcal{E}$-harmonic in $F(\xi)$, we have, for $\mathcal{E}$-q.e. $z\in B_{U}(\zeta,2r)$,   \begin{align} \label{e-g-cp-1} g_{D}(x,z)& \leq \omega(z,U\cap \partial B_{U}(\zeta,2r), B_{U}(\zeta,2r))\sup_{z^{\prime}\in U\cap \partial B_{U}(\zeta,2r)}g_{D}(x, z^{\prime})\ \text{(by Lemma \ref{l-g-hmk})} \\ &\leq \oldconstant{c6}\oldconstant{c31}\frac{g_{B_{U}(\zeta,A_{2}r)}(z,\zeta_{r})}{g_{B_{U}(\zeta,A_{2}r)}(\zeta_{r}^{\prime},\zeta_{r})} g_{D}(x, y_{\xi}^{\star})\ \text{(by Lemma \ref{l-hmkm} and Lemma \ref{l-g-cp-anu})} \\ &\leq \oldconstant{c6}\oldconstant{c31}\oldconstant{c10}^{2}\frac{g_{D}(x_{\xi}^{\star},z)}{g_{D}(x_{\xi}^{\star},y_{\xi}^{\star})} g_{D}(x, y_{\xi}^{\star}) \ \text{(by Lemma \ref{l-g-4po} with $\eta=\zeta_{r}$ and $\eta^{\prime}=\zeta_{r}^{\prime}$)}\end{align} By the continuity of $g_{D}(x,\cdot)$ and $g_{D}(x_{\xi}^{\star}, \cdot)$, we can extend \eqref{e-g-cp-1} to all $z\in B_{U}(\zeta,2r)$, in particular, for $y$.  Therefore \eqref{e-g-cp-u} holds with $\oldconstant{c32}=\oldconstant{c6}\oldconstant{c31}\oldconstant{c10}^{2}$.
\end{proof}
Combining Lemmas \ref{l-g-cp-bd}, \ref{l-g-cp-l} and \ref{l-g-cp-u}, we now have 
\begin{proposition}\label{p-2sdd}
    There exists $\newconstant\label{c33}$ such that for all $\xi\in\partial U$, $0<r< C_{0}^{-1}\diam(U,d)$, we have, by writing $D=B_{U}(\xi,A_{4}r)$,\begin{equation}\label{e-2sdd}
        \frac{g_{D}(x_{1},y_{1})}{g_{D}(x_{2},y_{1})}\leq \oldconstant{c33}\frac{g_{D}(x_{1},y_{2})}{g_{D}(x_{2},y_{2})}\quad\text{for all }x_{j}\in B_{U}(\xi,r),y_{j}\in U\cap\partial B_{U}(\xi,A_{3}r), j=1,2.
    \end{equation}
\end{proposition}

By Proposition \ref{p-2sdd} and representation formula in Proposition \ref{p-blyg}, we can adopt the proof of \cite[Theorem 5.2]{Lie15} to get Theorem \ref{t-bhp}.
\begin{proof}[Proof of Theorem \ref{t-bhp}]
By Proposition \ref{p-2sdd}, \[\frac{g_{D}(x_{1},y)}{g_{D}(x_{2},y)}\leq \oldconstant{c33}\frac{g_{D}(x_{1},y^{\prime})}{g_{D}(x_{2},y^{\prime})}\text{ for all }y,y^{\prime}\in U\cap\partial B_{U}(\xi,A_{3}r).\] According to the representation formula in Proposition \ref{p-blyg}, there exists a Radon measure $\nu_{u}$ and an $m$-version of $u$, denoted by $\widetilde{u}$, such that for $x_{1}$, $x_{2}\in B_{U}(\xi,r)$,\begin{align}\label{e-lp1}
    \widetilde{u}(x_{1})&=\int_{U\cap \partial B_{U}(\xi,A_{3}r)}g_{D}(x_{1},y)\dif \nu_{u}(y)\\
    &\leq \oldconstant{c33}\frac{g_{D}(x_{1},y^{\prime})}{g_{D}(x_{2},y^{\prime})} \int_{U\cap \partial B_{U}(\xi,A_{3}r)}g_{D}(x_{2},y)\dif \nu_{u}(y)=\oldconstant{c33}\frac{g_{D}(x_{1},y^{\prime})}{g_{D}(x_{2},y^{\prime})}\widetilde{u}(x_{2}).
\end{align}
Interchange $x_{1}$ and $x_{2}$, \begin{equation}\label{e-lp2}
    \widetilde{u}(x_{2})\leq \oldconstant{c33}\frac{g_{D}(x_{2},y^{\prime})}{g_{D}(x_{1},y^{\prime})}\widetilde{u}(x_{1}).
\end{equation}
Replace $u$ by $v$ in \eqref{e-lp2}, \begin{equation}\label{e-lp3}
    \widetilde{v}(x_{2})\leq \oldconstant{c33}\frac{g_{D}(x_{2},y^{\prime})}{g_{D}(x_{1},y^{\prime})}\widetilde{v}(x_{1}).
\end{equation}where $\widetilde{v}$ is an $m$-version of $v$ given by Proposition \ref{p-blyg}. Combining \eqref{e-lp1} and \eqref{e-lp3} we have \[\frac{\widetilde{u}(x_{1})}{\widetilde{u}(x_{2})}\leq \oldconstant{c33}\frac{g_{D}(x_{1},y^{\prime})}{g_{D}(x_{2},y^{\prime})}\leq \oldconstant{c33}^{2} \frac{\widetilde{v}(x_{1})}{\widetilde{v}(x_{2})}\quad\text{for all }x_{1},x_{2}\in B_{U}(\xi,r),\] 
which implies \[\esssup_{x_{1}\in B_{U}(\xi,r)}\frac{u(x_{1})}{v(x_{1})}\leq \oldconstant{c33}^{2} \essinf_{x_{2}\in B_{U}(\xi,r)}\frac{u(x_{2})}{v(x_{2})}.\]
\end{proof}
\vspace{10pt}

\noindent \textbf{Acknowledgments.} 
I am grateful to Mathav Murugan for proposing the problem tackled in this paper, for the reference \cite{KM23a}, for many helpful discussions and invaluable feedback throughout the writing process.

\vspace{11pt}
\noindent Department of Mathematical Sciences, Tsinghua University, Beijing 100084, China\\
and \\
Department of Mathematics, University of British Columbia, Vancouver, BC V6T 1Z2, Canada.

\vspace{3pt}
\noindent \texttt{aobochen.math@hotmail.com}

\vspace{-2pt}
\noindent \texttt{cab21@mails.tsinghua.edu.cn}

\end{document}